 \documentclass[12pt,oneside]{amsart}
\usepackage{amssymb,mathrsfs,amsfonts,longtable, hhline, textcomp}
\usepackage[matrix,arrow,curve]{xy}
\usepackage{url}

\makeatletter \@addtoreset{equation}{section} \makeatother

\newcommand{\muu}{{\boldsymbol{\mu}}}

\newcommand{\lcm}{\operatorname{lcm}}
\newcommand{\red}{\operatorname{red}}
\newcommand{\Pic}{\operatorname{Pic}}

\newcommand{\diff}{\operatorname{d}}
\newcommand{\ord}{\operatorname{ord}}

\newcommand{\Supp}{\operatorname{Supp}}

\newcommand{\Cl}{\operatorname{Cl}}
\newcommand{\ct}{\operatorname{ct}}

\newcommand{\ov}[1]{\overline{#1}}
\newcommand{\B}{{\mathbf B}}
\newcommand{\CC}{\mathbb{C}}
\newcommand{\QQ}{\mathbb{Q}}

\newcommand{\ZZ}{\mathbb{Z}}
\newcommand{\PP}{\mathbb{P}}
\newcommand{\OOO}{{\mathscr{O}}} 
\newcommand{\MMM}{{\mathscr{M}}} 
\newcommand{\EEE}{{\mathscr{E}}} 
 
\renewcommand{\emptyset}{\varnothing}

\newcounter{NO}

\newcommand{\qW}{\operatorname{qW}}

\newcommand{\qQ}{\operatorname{q\QQ}}

\newcommand{\qq}{\mathbin{\sim_{\scriptscriptstyle{\QQ}}}}

\newcommand{\comment}[1]{}

\newcommand{\xref}[1]{{\rm \ref{#1}}}

\newtheorem{theorem}[equation]{Theorem}
\newtheorem{proposition}[equation]{Proposition}
\newtheorem{lemma}[equation]{Lemma}

\newtheorem{corollary}[equation]{Corollary}

\newtheorem{claim}[equation]{Claim}

\theoremstyle{definition}
\newtheorem{remark}[equation]{Remark}

\newtheorem{case}[equation]{}

\title{$\QQ$-Fano threefolds of large Fano index, I}

\author{Yuri Prokhorov}
\thanks{The author was partially supported by 
the Russian Foundation for Basic Research
(grants
No~
06-01-72017-MNTI\_a,
08-01-00395-a) and Leading 
Scientific Schools (grants No NSh-1983.2008.1, NSh-1987.2008.1)
}

\address{Department 
of Algebra, Faculty of Mathematics, Moscow State
University, Moscow 117234, Russia}
\email{prokhoro@gmail.com}

\begin{document}
\begin{abstract}
We study $\QQ$-Fano threefolds of large Fano index.
In particular, we prove that the maximum of Fano index is attained 
for the weighted projective space $\PP(3,4,5,7)$.
\end{abstract}
\maketitle

\section{Introduction}
The Fano index of
a smooth Fano variety $X$ is the maximal integer $\operatorname{q}(X)$ 
that divides the anti-canonical class in the Picard group $\Pic(X)$
\cite{Iskovskikh-Prokhorov-1999}.
It is well-known \cite{Kobayashi-Ochiai-1973} that $\operatorname{q}(X)\le \dim X+1$. Moreover,
$\operatorname{q}(X)=\dim X+1$ if and only if $X$ is a projective space and
$\operatorname{q}(X)=\dim X$ if and only if $X$ is a quadric hypersurface.
In this paper we consider generalizations of Fano index
for the case of singular Fanos admitting terminal singularities.

A normal projective variety $X$ is called \textit{Fano} if some positive multiple 
$-nK_X$ of its anti-canonical Weil divisor is Cartier and ample.
Such $X$ is called a \textit{$\QQ$-Fano variety} if it has only 
terminal $\QQ$-factorial singularities and its Picard number is one.
This class of Fano varieties is important because they 
appear naturally in the Minimal Model Program.

For a singular Fano variety $X$ 
the Fano index can be defined in different ways. For example, we can define
\[
\begin{array}{lll}
\qW(X)&:=& \max \{ q \mid -K_X\sim q A,\quad \hbox to 0.42
\textwidth {\ \text{$A$ is a Weil $\QQ$-Cartier divisor}} \},
\\[8pt]
\qQ(X)&:=& \max \{ q \mid -K_X\qq q A, \quad \raise 2.9pt
\hbox to 0.42\textwidth {\quad \hrulefill \textquotestraightdblbase \hrulefill\quad }\}.
\end{array}
\]
If $X$ has at worst log terminal singularities, then 
the Picard group $\Pic(X)$ and Weil divisor class group $\Cl(X)$ 
are finitely generated and 
$\Pic(X)$ is torsion free 
(see e.g. \cite[\S 2.1]{Iskovskikh-Prokhorov-1999}).
Moreover, the numerical equivalence of
$\QQ$-Cartier divisors coincides with $\QQ$-linear one.
This implies, in particular, that defined above 
Fano indices $\qW(X)$ and $\qQ(X)$ are positive integers.
If $X$ is smooth, these numbers 
coincide with the Fano index  $\operatorname{q}(X)$ 
defined above. 
Note also that $\qQ(X)=\qW(X)$ if the group $\Cl(X)$ is torsion free.

\begin{theorem}[{\cite{Suzuki-2004}}]
\label{th-suzuki}
Let $X$ be a $\QQ$-Fano thereefold. Then $\qW(X)\in \{1,\dots,11,13,17,19\}$.
All these values, except possibly for $\qW(X)=10$, occur.
Moreover, if $\qW(X)=19$, then the types of non-Gorenstein points and
Hilbert series of $X$ coincide with that of $\PP(3,4,5,7)$.
\end{theorem}

It can be easily shown (see proof of Proposition \ref{prop-comput}) 
that the index $\qQ(X)$ takes values in the same set $\{1,\dots,11,13,17,19\}$. 
Thus one can expect that $\PP(3,4,5,7)$ is the only example 
of $\QQ$-Fano threefols with $\qQ(X)=19$. In general,
we expect that Fano varieties with extremal properties (maximal degree,
maximal Fano index, etc.) are quasihomogeneous with respect to 
an action of some connected algebraic group. This is supported, for example, 
by the following facts:

\begin{theorem}[\cite{Prokhorov-2005a}, {\cite{Prokhorov-2007-Qe}}]
\label{th-degree}
\begin{enumerate}
 \item 
Let $X$ be a $\mathbb Q$-Fano threefold. 
Assume that~$X$ is not Gorenstein.
Then $-K_X^3\le125/2$ and the equality holds 
if and only if~$X$ is isomorphic to the weighted projective space 
$\PP(1^3,2)$.

 \item 
Let $X$ be a Fano threefold with canonical Gorenstein singularities.
Then $-K_X^3\le 72$ and the equality holds 
if and only if~$X$ is isomorphic to 
$\PP(1^3,3)$ or $\PP(1^2,6, 4)$.
\end{enumerate}
\end{theorem}

The following proposition is well-known (see, e.g., \cite{Borisov-Borisov}).
It is an easy exercise for experts in toric geometry.

\begin{proposition}
\label{proposition-main-toric}
Let $X$ be a toric $\QQ$-Fano $3$-fold. Then $X$ is isomorphic 
to either $\PP^3$, $\PP^3/\muu_5(1,2,3)$, or one of the following weighted projective spaces:
\par\medskip\noindent
$\PP(1^3,2)$,\hspace{4pt}
$\PP(1^2,2,3)$,\hspace{4pt}
$\PP(1,2,3,5)$,\hspace{4pt}
$\PP(1,3,4,5)$,\hspace{4pt}
$\PP(2,3,5,7)$,\hspace{4pt}
$\PP(3,4,5,7)$.
\end{proposition}

We characterize the weighted projective spaces above 
in terms of Fano index.
The following is the main result of this paper.

\begin{theorem}
\label{theorem-main}
Let $X$ be a $\QQ$-Fano threefold. Then $\qQ(X)\in \{1,\dots,11,13,17,19\}$.
\begin{enumerate}
\item\label{theorem-main-19}
If $\qQ(X)=19$, then $X\simeq \PP(3,4,5,7)$.
\item\label{theorem-main-17}
If $\qQ(X)=17$, then $X\simeq \PP(2,3,5,7)$.

\item\label{theorem-main-13}
If $\qQ(X)=13$ and $\dim |-K_X| >5$, then $X\simeq \PP(1,3,4,5)$.

\item\label{theorem-main-11}
If $\qQ(X)=11$ and $\dim |-K_X| >10$, then $X\simeq \PP(1,2,3,5)$.

\item\label{theorem-main-10}
$\qQ(X)\neq 10$.

\item\label{theorem-main-q7}
If $\qQ(X)\ge 7$ and 
there are two effective Weil divisors $A\neq A_1$ such that
$-K_X\qq \qQ(X)A\qq \qQ(X)A_1$, then $X\simeq \PP(1^2,2,3)$.

\item\label{theorem-main-q5}
If $\qW(X)=5$ and $\dim |-\frac15K_X| >1$, then $X\simeq \PP(1^3,2)$.
\end{enumerate}
\end{theorem}

Note that in cases \ref{theorem-main-13} and \ref{theorem-main-11}
assumptions about $\dim |-K_X|$ are needed. Indeed, 
there are examples of non-toric $\QQ$-Fano threefolds 
with $\qQ(X)=13$ and $11$.

In the  proof we follow the use some techniques developed in our previous paper
\cite{Prokhorov-2007-Qe}.
By Proposition \ref{proposition-main-toric} it is sufficient to
show that our $\QQ$-Fano $X$ is toric. 
First, as in \cite{Suzuki-2004}, we apply the orbifold Riemann-Roch
formula to find all the possibilities for the numerical 
invariants of $X$.
In all cases there is some special element $S\in |-K_X|$
having four irreducible components.
This $S$ should be a toric boundary, if $X$ is toric.
Further, we use birational transformations like 
Fano-Iskovskikh ``double projection'' \cite{Iskovskikh-Prokhorov-1999}
(see \cite{Alexeev-1994ge} for the $\QQ$-Fano version).
Typically the resulting variety is
a Fano-Mori fiber space having ``simpler'' structure.
(In particular, its Fano index is large if this variety is a $\QQ$-Fano).
By using properties of our ``double projection'' 
we can show that the pair $(X,S)$ is log canonical (LC).
Then,
in principle, the assertion follows by 
Shokurov's toric conjecture \cite{McKernan-2001}.
We prefer to propose an alternative, more explicit proof.
In fact, the image of $X$ under ``double projection''
is a toric variety and the inverse map preserves the toric structure. 
In the last section we describe Sarkisov links between toric $\QQ$-Fanos
that starts with blowing ups  \textit{singular} points.

\subsection*{Acknowledgements}
The work was conceived during the authors stay 
at the University of Warwick in the spring of 2008.
The author would like to thank Professor M. Reid for
invitation, hospitality and fruitfull discussions. 
Part of the work was done at 
Max-Planck-Institut f\"ur Mathematik, Bonn in August 2008.

\section{Preliminaries, the orbifold Riemann-Roch formula and its applications}
\label{sect-2}
\subsection*{Notation}
Throughout this paper, we work over the complex number field $\CC$.
We employ the following standard notation:


$\sim$\quad denotes the linear equivalence;

$\qq$\quad denotes the $\QQ$-linear equivalence.

Let $E$ be a rank one discrete valuation of the function field $\CC(X)$ and let $D$ is a $\QQ$-Cartier 
divisor on $X$. $a(E,D)$ \ 
denotes the discrepancy of $E$ with respect to a boundary $D$.
Let $f\colon \tilde X\to X$ be a birational morphism such that
$E$ appears as a prime divisor on $\tilde X$. Then 
$\ord_E(D)$ denotes the coefficient of $E$ in $f^*D$.

\begin{case} \textbf{The orbifold Riemann-Roch formula \cite{Reid-YPG1987}.}
Let $X$ be a threefold with terminal singularities 
and let $D$ be a Weil $\QQ$-Cartier divisor on $X$.
Let $\B=\{ (r_P,b_P)\}$ be the basket of singular points of $X$ \cite{Mori-1985-cla}, \cite{Reid-YPG1987}.
Here each pair $(r_P,b_P)$ correspond to a point $P\in \B$ of type $\frac1{r_P}(1,-1,b_P)$.
For brevity, describing a basket we will list just indices of singularities, 
i.e., we will write $\B=\{ r_P\}$ instead of $\B=\{ (r_P,b_P)\}$.
In the above situation, the Riemann-Roch formula has the following form
\begin{multline}
\label{eq-RR}
\chi(D)=
\frac1{12}D\cdot (D-K_X)\cdot (2D-K_X)+
\\
+\frac1{12}D\cdot c_2+\sum_{P\in \B} c_P(D)+ 
\chi(\OOO_X),
\end{multline}
where 
\[
c_P(D)=-i_P\frac{r_P^2-1}{12r_P}+\sum_{j=1}^{i_P-1}\frac{\ov{b_Pj}(r_P-\ov{b_Pj})}{2r_P}.
\]
Clearly, computing $c_P(D)$, we always may assume that $1\le b_P\le r_P/2$.
\end{case}

\begin{case}
\label{not-chi}
Now let $X$ be a Fano threefold with terminal singularities, 
let $q:=\qQ(X)$, and
let $A$ be an ample Weil $\QQ$-Cartier divisor on $X$ such that $-K_X\qq qA$. 
By \eqref{eq-RR} we have
\begin{equation}
\label{eq-comput-chi}
\chi(tA)=1+\frac{t(q+t)(q+2t)}{12}A^3+
\frac{tA\cdot c_2}{12}
+ \sum_{P\in \B} c_P(tA),
\end{equation}
\[
c_P(tA)=
-i_{P, t} \frac{r_P^2-1}{12r_P}+\sum_{j=1}^{i_{P, t}-1}
\frac{\ov{b_Pj}(r_P-\ov{b_Pj})}{2r_P}.
\]
If $q>2$, then $\chi(-A)=0$. Using this equality
we obtain (see \cite{Suzuki-2004})
\begin{equation}
\label{eq-comput-L3}
A^3=\frac{12}{(q-1)(q-2)}\left( 
1-\frac{A\cdot c_2}{12}+\sum_{P\in B} c_P(-A)
\right).
\end{equation}
\end{case}

In the above notation,
applying \eqref{eq-RR}, Serre duality and Kawamata-Viehweg vanishing 
to $D=K_X$, we get the following important equality (see, e.g., \cite{Reid-YPG1987}):
\begin{equation}
\label{eq-RR-O}
24=-K_X\cdot c_2(X)+\sum_{P\in \B} \left( r_P-\frac 1{r_P}\right).
\end{equation}
\begin{theorem}[{\cite{Kawamata-1992bF}}, {\cite{KMMT-2000}}]
\label{th-Kawamata-Kc}
In the above notation, 
\begin{equation}
\label{eq-bogomolov}
-K_X\cdot c_2(X)\ge 0,\qquad \sum_{P\in \B} \left( r_P-\frac 1{r_P}\right)\le 24.
\end{equation}
\end{theorem}

\begin{proposition}
\label{lemma-torsion}
Let $X$ be a Fano thereefold with terminal 
singularities and let $\Xi$ be an $n$-torsion element in the Weil divisor
class group. 
Let $\B^\Xi$ be the collection of points $P\in \B$ where $\Xi$ is not Cartier.
Then
\begin{equation}
\label{eq-torsion-}
2=\sum_{P\in \B^{\Xi}}
\frac{\ov{b_Pi_{\Xi,P}}\ \left(r_P-\ov{b_Pi_{\Xi,P}}\right)}{2r_P}.
\end{equation}
where $i_{\Xi,P}$ is taken so that $\Xi\sim i_{\Xi,P}K_X$ near $P\in \B$
and $\ov{\phantom{rt}}$ is the residue $\mod r_P$.
Assume furthermore that $n$ is prime. Then
\begin{enumerate}
\item
$n\in \{2,3,5,7\}$.
\item
If $n=7$, then $\B^{\Xi}=(7,7,7)$.\footnote{More 
delicate computations show that this case does not occur.
(We do not need this.)}
\item
If $n=5$, then $\B^{\Xi}=(5,5,5,5)$, $(10,5,5)$, or $(10,10)$.
 \item 
If $n=3$, then $\sum_{P\in \B^{\Xi}} r_P=18$.
 \item 
If $n=2$, then $\sum_{P\in \B^{\Xi}} r_P=16$.
\end{enumerate}
\end{proposition}

\begin{proof}
Let By Riemann-Roch \eqref{eq-RR}, Kawamata-Viehweg vanishing theorem and
Serre duality we have
\[
\begin{array}{lll}
0&=\chi(\Xi)&=1+\sum _P c_P(\Xi),
\\[10pt]
0&=\chi(K_X+\Xi)&=1+\frac1{12}K_X\cdot c_2(X)+\sum_{P\in \B} c_P(K_X+\Xi).
\end{array}
\]
Subtracting we get
\[
0=-\frac1{12}K_X\cdot c_2(X)+\sum_{P\in \B} (c_P(\Xi)-c_P(K_X+\Xi)).
\]
Since $ni_{\Xi,P}\equiv 0\mod r_P$,
\[
0=-\frac1{12}K_X\cdot c_2(X)+\frac {1}{12}
\sum_{P\in \B} \left( r_P-\frac{1}{r_P}\right) -
\sum_{P\in \B}
\frac{\ov{b_Pi_{\Xi,P}}\ \left(r_P-\ov{b_Pi_{\Xi,P}}\right)}{2r_P}.
\]
This proves \eqref{eq-torsion-}.

Now assume that $n$ is prime. 
If $P\in \B^\Xi$, then $n\mid r_P$. Write $r_P=nr'_P$. Since $r_P\mid ni_P$,
$i_P=r'_Pi'_P$, where $n \nmid i'_P$. 
Let $\ov{(\phantom{nm})_n}$ be the residue $\mod n$.
Then 
\[
2=\sum_{P\in \B^\Xi}
\frac{\ov{b_Pi_{\Xi,P}'r'}\ \left(nr'_P-\ov{b_Pi_{\Xi,P}'r'_P}\right)}{2nr_P'}=
\frac{r_P^{\prime}\ov{(b_Pi_{\Xi,P}')_n}\ 
\left(n-\ov{(b_Pi_{\Xi,P}')_n}\right)}{2n}.
\]
Therefore,
\begin{equation*}
\label{eq-torsion-2}
4n^2=\sum_{P\in \B^{\Xi}} r_P\ov{(b_Pi_{\Xi,P}')_n}\ \left(n-\ov{(b_Pi_{\Xi,P}')_n}\right).
\end{equation*}
Denote $\xi_P:=\ov{(b_Pi_{\Xi,P}')_n}$.
Then $0<\xi_P< n$, $\gcd(n, \xi_P)=1$, and
\[
4n=\sum_{P\in \B^{\Xi}} r_P'\xi_P \left(n-\xi_P\right)\ge \frac{n^2}{4}\sum_{P\in \B^{\Xi}} r_P',
\qquad 16\ge n\sum_{P\in \B^{\Xi}} r_P'.
\]
If $n\ge 11$, then $\sum r_P'=1$,  $n\mid r_P'$, 
and $r_P\ge n^2\ge 121$, a contradiction.
Therefore, $n\le 7$.
Consider the case $n=7$. Then 
$\xi_P\left(n-\xi_P\right)=6$, $10$, or $12$.
The only solution is $\B^{\Xi}=(7,7,7)$.
The case $n=5$ is considered similarly.
If  $n=3$, then $\xi_P \left(n-\xi_P\right)=3$ and $\sum r_P=3\sum r_P'=18$.
Similarly, if $n=2$, then $\xi_P \left(n-\xi_P\right)=1$ and $\sum r_P=2\sum r_P'=16$.
This finishes the proof.
\end{proof}

\section{Computations with Riemann-Roch on $\QQ$-Fano threefolds of large Fano index}
\begin{lemma}[see \cite{Suzuki-2004}]
\label{lemma-fanoindex}
Let $X$ be a Fano threefold with terminal singularities
with $q:=\qW(X)$, let $A:=-\frac1qK_X$, and let $r$ be the Gorenstein index of $X$.
Then
\begin{enumerate}
 \item 
$r$ and $q$ are coprime;
\item
$rA^3$ is an integer.
\end{enumerate}
\end{lemma}

\begin{lemma}
\label{lemma-fano-index}\label{lemmaqWqQ-coprime}
Let $X$ be a Fano threefold with terminal singularities.
\begin{enumerate}
\item 
If $-K_X\sim q L$ for some Weil divisor $L$, then $q$ divides $\qW(X)$.
\item 
If $-K_X\qq q L$ for some Weil divisor $L$, then $q$ divides $\qQ(X)$.
\item 
$\qW(X)$ divides $\qQ(X)$.
\item 
Let $q:=\qQ(X)$
and let $K_X+qA\qq 0$.
If the order of $K_X+qA$ in the group 
$\Cl(X)$ is prime to $q$, then $\qW(X)=\qQ(X)$.
\end{enumerate}
\end{lemma}

\begin{proof}
To prove (i) write $-K_{X}\sim \qW(X) A$ and 
let $d=\gcd(\qW(X),q)$. Then $d=u\qW(X)+v q$ for some 
$u,\, v\in \ZZ$. Hence, $dA=u\qW(X)A+v q A\sim 
quL+qvA=q(uL+vA)$.
Since $A$ is a primitive element of $\Cl(X)$,
$q=d$ and $q \mid \qW(X)$. 

(ii) can proved similarly and (iii) is a 
consequence of (ii).

To show (iv) assume that $\Xi:=K_{X}+qA$ is of order $n$.
By our condition $qu+nv=1$, where $u,v\in \ZZ$.
Put $A':=A-u\Xi$. 
Then $qA'=qA-qu\Xi=qA-\Xi \sim -K_X$. Hence, $q=\qW(X)$ by (i) and (iii).
\end{proof}

\begin{lemma}
\label{lemma-fano}
Let $X$ be a Fano threefold with terminal singularities.
\begin{enumerate}
 \item 
$\qQ(X)\in \{1,\dots,11, 13,17,19\}$.
 \item 
If $\qQ(X)\ge 5$, then $-K_X^3\le 125/2$.
\end{enumerate}
\end{lemma}
\begin{proof}
Denote $q:=\qQ(X)$ and write, as usual, $-K_X\qq qA$.
Thus $n(K_X+qA)\sim 0$ for some positive integer $n$.
The element $K_X+qA$ defines a cyclic \'etale in codimension one 
cover $\pi\colon X'\to X$ so that $X'$ is a Fano threefold with 
terminal singularities and $K_{X'}+qA'\sim 0$, where $A':=\pi^*A$.
Let  $\sigma\colon X''\to X'$ be a $\QQ$-factorialization.
(If $X'$ is $\QQ$-factorial, we take $X''=X'$).
Run $K$-MMP on $X''$:\quad $\psi\colon X'' \dashrightarrow \bar X$.
At the end we get a Mori-Fano fiber space $\bar X\to Z$.
Let $A'':=\sigma^{-1}(A')$ and $\bar A:=\psi_*A''$.
Then $-K_{\bar X}\sim q\bar A$. 
If $\dim Z>0$, then for a general fiber $F$ of $\bar X/Z$,
we have $-K_F\sim q \bar A|_F$. This is impossible because $q>3$. 
Thus $\dim Z=0$ and $\bar X$ is a $\QQ$-Fano.

(i) By Lemma \ref{lemma-fano-index} the number $q$ divides $\qW(\bar X)$.
On the other hand, by Theorem \ref{th-suzuki} we have $\qW(\bar X)\in \{1,\dots,11, 13,17,19\}$.
This proves (i). 

To show (ii) we note that $-K_{\bar X}^3\ge-K_{X''}^3=-K_{X'}^3\ge -K_{X''}^3$.
Here the first inequality holds because for Fanos
(with at worst  log terminal singularities) the number $-\frac16 K^3$ is nothing but the leading term 
in the asymptotic Riemann-Roch and $\dim |-tK_{X''}|\le \dim |-tK_{\bar X}|$.
Now the assertion of (ii) follows from Theorem \ref{th-degree}.
\end{proof}

From Lemmas \ref{lemma-fano-index} and \ref{lemma-fano} we have
\begin{corollary}
Let $X$ be a Fano threefold with terminal singularities. 
\begin{enumerate}
\item 
If $-K_X\sim qL$ 
for some Weil divisor $L$ and $q\ge 5$, then $q=\qW(X)$.
\item 
If $-K_X\qq qL$ 
for some Weil divisor $L$ and $q\ge 5$, then $q=\qQ(X)$.
\end{enumerate}
\end{corollary}

\begin{lemma}[cf. \cite{Suzuki-2004}]
\label{lemma-comput}
Let $X$ be a Fano threefold with terminal singularities and 
let $q:=\qW(X)$. Assume that
$\qW(X)\ge 8$.
Then one of the following holds:
\begin{itemize}
 \item[]
$q=8$,\quad
$\B=(3^2, 5)$, 
$(3^2, 5, 9)$, 
$(3, 5, 11)$, 
$(3, 7)$, 
$(3, 9)$, 
$(5, 7)$, 
$(7, 11)$, 
$(7, 13)$, 
$(11)$,
 \item []
$q=9$,\quad
$\B=(2, 4, 5)$, 
$(2^3, 5, 7)$,
$(2, 5,13)$,
 \item []
$q=10$,\quad
$\B=(7, 11)$,
 \item []
$q=11$,\quad
$\B=(2, 3, 5)$, 
$(2, 5, 7)$, 
$(2^2, 3, 4, 7)$,
 \item []
$q=13$,\quad
$\B=(3, 4, 5)$,
$(2, 3^2, 5, 7)$,
 \item []
$q=17$,\quad
$\B=(2, 3, 5, 7)$,
 \item []
$q=19$,\quad
$\B=(3, 4, 5, 7)$.
\end{itemize}
In all cases the group $\Cl(X)$ is torsion free.
\end{lemma}

\begin{proof}
We use a computer program written in PARI \cite{PARI2}. 
Below is the description of our algorithm.

\textbf{Step 1.}
By Theorem \ref{th-Kawamata-Kc} we have 
$\sum_{P\in \B} (1-1/r_P)\le 24$. Hence there is only 
a finite (but very huge) number of possibilities for the basket 
$\B=\{[r_P,b_P]\}$.
In each case we know $-K_X\cdot c_2(X)$ from 
\eqref{eq-RR-O}.
Let $r:=\lcm (\{r_P\})$ be the Gorenstein index of $X$.

\textbf{Step 2.}
By Lemma \ref{lemma-fano}
$\qQ(X)\in \{8,\dots,11, 13,17,19\}$.
Moreover, the condition $\gcd (q,r)=1$
(see Lemma \ref{lemma-fanoindex})
eliminates some possibilities.

\textbf{Step 3.}
In each case we compute $A^3$ and $-K_X^3=q^3A^3$ by formula
\eqref{eq-comput-L3}.
Here, for $D=-A$, the number $i_P$ is uniquely determined by
$qi_P\equiv b_P\mod r_P$ and $0\le i_P< r_P$.
Further, we check the condition $rA^3\in \ZZ$
(Lemma \ref{lemma-fanoindex})
and the inequality $-K_X^3\le 125/2$
(Lemma \ref{lemma-fano}).

\textbf{Step 4.}
Finally, by the Kawamata-Viehweg vanishing theorem
we have $\chi(tA)=h^0(tA)$ for $-q<t$.
We compute $\chi(tA)$ by using 
\eqref{eq-comput-chi} and check conditions 
$\chi(tA)=0$ for $-q<t<0$ and $\chi(tA)\ge 0$ for $t>0$.

At the end we get our list.
To prove the last assertion 
assume that $\Cl(X)$ contains an $n$-torsion element $\Xi$. 
Clearly, we also may assume that $n$ is prime.
By Proposition \ref{lemma-torsion} we have $\sum_{n \mid r_i} r_i\ge 16$.
Moreover, $\sum_{n \mid r_i} r_i\ge 18$ if $n=3$.
This does not hold in all cases of our list.
\end{proof}

\begin{proposition}
\label{prop-comput}
Let $X$ be a $\QQ$-Fano threefold with
$\qQ(X)\ge 9$.
Let $q:=\qQ(X)$ and let $-K_X\qq q A$.
Then the group $\Cl(X)$ is torsion free, $\qW(X)=\qQ(X)$,
 and one of the following holds:
\rm
\par\medskip\noindent
\setlongtables\renewcommand{\arraystretch}{1.3}
\begin{longtable}{|c|c|c|c|c|c|c|c|c|c|c|}
\hline
&&&\multicolumn{8}{c|}{$\dim |kA|$}
\\
\hhline{|~|~|~|--------}
$q$&$\B$&$A^3$& $|A|$&$|2A|$&$|3A|$&$|4A|$&$|5A|$&$|6A|$&$|7A|$&$|-K|$
\\[5pt]
\hline
\endfirsthead
\hline
$q$&$\B$&$A^3$& $|A|$&$|2A|$&$|3A|$&$|4A|$&$|5A|$&$|6A|$&$|7A|$&$|-K|$
\\[5pt]
\hline
\endhead
\hline
\endlastfoot
\hline
\endfoot

9
& $(2, 4, 5)$
&$\frac{1}{20}$
&0
&1
&2
&4
&6
&8
&11
&19
\\

9
& $(2, 2, 2, 5, 7)$
&$\frac{1}{70}$
&$-1$
&0
&0
&1
&1
&2
&3
&5

\\
10
& $(7, 11)$
&$\frac{2}{77}$
&$-1$
&0
&1
&1
&3
&4
&6
&13

\\
11
& $(2, 3, 5)$
&$\frac{1}{30}$
&0
&1
&2
&3
&5
&7
&9
&23

\\
11
& $(2, 5, 7)$
&$\frac{1}{70}$
&0
&0
&0
&1
&2
&3
&4
&10

\\
11
& $(2, 2, 3, 4, 7)$
&$\frac{1}{84}$
&$-1$
&0
&0
&1
&1
&2
&3
&8

\\
13
& $(3, 4, 5)$
&$\frac{1}{60}$
&0
&0
&1
&2
&3
&4
&5
&19

\\
13
& $(2, 3, 3, 5, 7)$
&$\frac{1}{210}$
&$-1$
&$-1$
&0
&0
&0
&1
&1
&5

\\
17
& $(2, 3, 5, 7)$
&$\frac{1}{210}$
&$-1$
&0
&0
&0
&1
&1
&2
&12

\\
19
& $(3, 4, 5, 7)$
&$\frac{1}{420}$
&$-1$
&$-1$
&0
&0
&0
&0
&1
&8
\end{longtable}
\end{proposition}

\begin{proof}
First we claim that $\qW(X)=\qQ(X)$.
Assume the converse. 
Then, as in the proof of Lemma \ref{lemma-fano}, the class of 
$K_X+qA$ is a non-trivial $n$-torsion element in $\Cl(X)$
defining a global cover $\pi\colon X'\to X$.
We have $K_{X'}+qA'\sim 0$, where $A'=\pi^*A$.
Hence $X'$ is such as in Lemma \ref{lemma-comput} 
and by Corollary \ref{lemma-comput}
we have $\Cl(X')\simeq \ZZ\cdot A'$ and $\qW(X')=\qQ(X')\ge q$.
The Galois group $\muu_n$ acts naturally on $X'$.
Consider, for example, the case $q=11$ and $\B_{X'}=(2,3,5)$
(all other cases are similar).
Then $X'$ has three cyclic quotient singularities whose
indices are $2$, $3$, and $5$. These points must be $\muu_n$-invariant.
Hence the variety $X$ has cyclic quotient singularities of indices 
$2n$, $3n$, and $5n$. 
By Lemma \ref{lemmaqWqQ-coprime} we have $\gcd(q,n)\neq 1$.
In particular, $n\ge 11$.
This contradicts \eqref{eq-bogomolov}.
Therefore, $\qW(X)=\qQ(X)$ and so $X$ is such as in Lemma \ref{lemma-comput}.

Now we have to exclude only the case  $q=9$, $\B=(2,5,13)$.
But in this case by \eqref{eq-RR-O} and \eqref{eq-comput-L3} we have 
$A^3=9/130$ and $-K_X\cdot c_2=621/130$.
On the other hand, by Kawamata-Bogomolov's bounds \cite{Kawamata-1992bF}
we have $2673/130=(4q^2-3q)A^3\le 4K_X\cdot c_2=1242/65$ \cite[Proposition 2.2]{Suzuki-2004}.
The contradiction shows that this case is impossible.
Finally, the values of $A^3$ and dimensions of $|kA|$ 
are computed by using \eqref{eq-comput-L3} and \eqref{eq-comput-chi}.
\end{proof}

\begin{corollary}
\label{corollary-cyclic}
Let $X$ be a $\QQ$-Fano threefold satisfying assumptions of
\textup{\ref{theorem-main-19}-\ref{theorem-main-10}} of Theorem \xref{theorem-main}. Then $X$ has only 
cyclic quotient singularities.
\end{corollary}
\begin{proof}
Indeed, in these cases the indices of points in the basket $\B$
are distinct numbers and moreover $\B$ contains no 
pairs of points of indices $2$ and $4$.
Then the assertion follows  \cite{Mori-1985-cla}, or \cite{Reid-YPG1987} 
\end{proof}

\begin{corollary}
\label{cor-dim-L}
Let $X$ be a $\QQ$-Fano threefold with $\qQ(X)\ge 9$.
Then $\dim |A|\le 0$.
\end{corollary}

Computer computations similar to that in 
Lemma \ref{lemma-comput} allow us to prove the following.

\begin{lemma}
\label{lemma-7}
Let $X$ be a Fano threefold with terminal singularities, let $q:=\qW(X)$,
and let $A:=-\frac1qK_X$.
\begin{enumerate}
\item 
If $q\ge 5$ and $\dim |A|>1$,
then $q=5$, $\B=(2)$, and $A^3=1/2$.

\item 
If $q\ge 7$ and $\dim |A|>0$, then $q=7$,
$\B=(2, 3)$,
 $A^3=1/6 $.
\end{enumerate}
\end{lemma}

\begin{case}\label{subsection-proof-q=5and7}
\textbf{Proof of \ref{theorem-main-q7} and 
\ref{theorem-main-q5} of Theorem \ref{theorem-main}.}
\ref{theorem-main-q5} Apply Lemma \ref{lemma-7}.
Then the result is well-known: in fact, $2A$ is Cartier and by Riemann-Roch
$\dim |2A|=6=\dim X+3$. Hence $X$ is a variety of $\Delta$-genus zero 
\cite{Fujita-1975}, i.e., a variety of minimal degree.
Then $X\simeq \PP(1^3,2)$.

\ref{theorem-main-q7} 
Put $q:=\qQ(X)$, $\Xi:=K_X+qA$, and $\Xi_1:=A-A_1$.
By our assumption $n\Xi\sim n\Xi_1\sim 0$ for some integer $n$.
If either $\Xi\not\sim 0$ or  $\Xi_1\not\sim 0$, then  
elements $\Xi$ and $\Xi'$ define 
an \'etale in codimension one finite cover
$\pi\colon X'\to X$ such that $K_{X'}+qA'\sim 0$ and $A'\sim A_1'$,
where $A':=\pi^*A$ and $A_1':=\pi^*A_1$.
If $\Xi\sim \Xi_1\sim 0$, we put $X'=X$.
In both cases, the following enequalities hold: $\qW(X')\ge 7$ and $\dim |A'|\ge 1$.
By Lemma \ref{lemma-7} we have $\B(X')=(2,3)$ and 
$\qQ(X')=\qW(X')=7$.
Note that the Gorenstein index of $X'$ is strictly less than
$\qW(X')$. In this case, $X'\simeq \PP(1^2,2,3)$ according to 
\cite{Sano-1996}.
\footnote{
The result also can be easily  proved by using birational transformations 
similar to that in \S \ref{sect-bir-constr}.}
Now it is sufficient to show that $\pi$ is an isomorphism.
Assume the converse.
By our construction, there is an action of a cyclic group $\muu_p\subset\operatorname{Gal}(X'/X)$,
$p$ is prime, such that $\pi$ is decomposed as $\pi\colon X'\to X'/\muu_p\to X$. Here 
$X'/\muu_p$ is a $\QQ$-Fano threefold and there is a torsion element 
of $\Cl(X'/\muu_p)$ which is not Cartier exactly at points where $X'\to X'/\muu_p$
is not \'etale. There are exactly four such points and two of them are 
points of indices $2$ and $3$. Thus the basket of $X'/\muu_p$ consists 
of points of indices $p$, $p$, $2p$, and $3p$.
This contradicts Proposition ~\ref{lemma-torsion}.
\end{case}

\begin{lemma}
\label{proposition-main}
Let $X$ be a $\QQ$-Fano threefold with $q:=\qQ(X)$.
If 
there are three effective different Weil divisors $A$, $A_1$, $A_2$ such that
$-K_X\qq qA\qq qA_1\qq qA_2$ and $A\not \sim A_1$, then 
$q\le 5$.
\end{lemma}

\begin{proof}
Assume that $q\ge 6$.
As in \ref{subsection-proof-q=5and7} consider a cover
$\pi\colon X'\to X$. Thus on $X'$ we have $A'\sim A_1'\sim A_2'$ 
and $-K_{X'}\sim qA'$. Moreover, $\dim |A'|=1$ according to Lemma ~ \xref{lemma-7}.
In this case, the action of $\operatorname{Gal}(X'/X)$ on the pencil
$|A'|$ is trivial (because there are three invariant members 
$A'$, $A_1'$, and $A_2'$). But then $A\sim A_1\sim A_2$,
a contradiction.
\end{proof}

\section{Birational construction}
\label{sect-bir-constr}
\begin{case}
Let $X$ be a $\QQ$-Fano threefold and let $A$ be the
ample Weil divisor that generates the group $\Cl(X)/{\qq}$. 
Thus we have $-K_X \qq qA$.
Let $\MMM$ be a mobile linear system 
without fixed components and let $c:=\ct(X,\MMM)$ be the canonical threshold of $(X,\MMM)$.
So the pair $(X,c\MMM)$ is canonical but not terminal.
Assume that $-(K_X+c\MMM)$ is ample.
\end{case}

Recall that the class of $K_X$ is a generator of the 
local Weil divisor class group $\Cl(X,P)$.
\begin{lemma}
\label{lemma-ct}
Let $P\in X$ be a point of index $r>1$.
Assume that $\MMM\sim -mK_X$ near $P$, where $0<m<r$.
Then $c\le 1/m$. 
\end{lemma}
\begin{proof}
According to \cite{Kawamata-1992-discr} there is an exceptional divisor $\Gamma$ over $P$ 
of discrepancy $a(\Gamma)=1/r$. Let $\varphi\colon Y\to X$ be a resolution. Clearly,
$\Gamma$ is a prime divisor on $Y$.
Write 
\[
K_Y=\varphi^*K_X+\frac1r \Gamma+\sum \delta_i\Gamma_i, 
\quad 
\MMM_Y=\varphi^*\MMM-\ord_\Gamma(\MMM)\Gamma-\ord_{\Gamma_i}(\MMM)\Gamma_i,
\]
where $\MMM_Y$ is the birational transform of $\MMM$ and $\Gamma_i$ 
are other $\varphi$-exceptional divisors.
Then 
\[
K_Y+c\MMM_Y=\varphi^*(K_X+cM)+(1/r -c\ord_\Gamma(\MMM))\Gamma+\dots
\]
and so $1/r -c\ord_\Gamma(\MMM)\ge 0$. On the other hand,
$\ord_\Gamma(\MMM) \equiv m/r\mod \ZZ$ (because $mK_X+\MMM\sim 0$ near $P$). 
Hence, $\ord_\Gamma(\MMM)\ge m/r$ and $c\le 1/m$.
\end{proof}

\begin{case} In the construction below we follow \cite{Alexeev-1994ge}.
Let $f\colon \tilde X\to X$ be $K+c\MMM$-crepant blowup such that 
$\tilde X$ has only terminal $\QQ$-factorial singularities:
\begin{equation}
\label{eq-Sarkisov-discr}
K_{\tilde X}+c\tilde \MMM=f^*(K_X+c\MMM).
\end{equation}
As in \cite{Alexeev-1994ge}, we run $K+c\MMM$-MMP on $\tilde X$. 
We get the following diagram \textup(Sarkisov link of type I or II\textup)
\begin{equation}
\label{eq3.1}
\begin{gathered}
\xymatrix{
&\tilde X\ar@{-->}[r]\ar[dl]_{f}&\bar X\ar[dr]^{g}&
\\
X&&&\hat X
} 
\end{gathered}
\end{equation}
where varieties~$\tilde X$ and~$\bar X$ have only $\mathbb Q$-factorial terminal
singularities, $\rho(\tilde X)=\rho (\bar X)=2$, 
$f$ is a Mori extremal divisorial contraction,
$\tilde X \dashrightarrow \bar X$ is a sequence of 
log flips, and $g$ is a Mori extremal contraction
(either divisorial or fiber type). 
Thus one of the following possibilities holds:
\begin{itemize}
\item[a)]
$\dim \hat X=1$ and $g$ is a $\mathbb Q$-del Pezzo fibration;
\item[b)]
$\dim \hat X=2$ and $g$ is a $\mathbb Q$-conic bundle; or
\item[c)]
$\dim \hat X=3$, $g$ is a divisorial contraction, and 
$\hat X$ is a $\mathbb Q$-Fano threefold.
In this case, denote $\hat q:=\qQ(\hat X)$.
\end{itemize}

Let $E$ be the 
$f$-exceptional divisor. 
For a divisor $D$ on $X$,
everywhere below~$\tilde D$ and $\bar D$ denote 
strict birational transforms of $D$ on~$\tilde X$ and $\bar X$, respectively.
If $g$ is birational, we put $\hat D:=g_*\bar D$.
\end{case}

\begin{claim}[\cite{Alexeev-1994ge}]
\label{st3.1}
If the map is birational, then $\bar E$ is not an exceptional divisor.
If $g$ is of fiber type, then 
$\bar E$ is not composed of fibers.
\end{claim}

\begin{proof}
Assume the converse.
If~$g$ is birational, this implies that the map
$g\circ\chi\circ f^{-1}\colon X\dashrightarrow \hat X$
is an
isomorphism in codimension~1.
Since both~$X$ and~$\hat X$ are Fano threefolds, this implies that
$g\circ\chi\circ f^{-1}$~is in fact an
isomorphism. 
On the other hand, the number of $K+c\MMM$-crepant divisors on
$\hat X$ is less than that on $X$, a contradiction.
If $\dim \hat X\le2$, then $\bar E$ is a pull-back of an ample Weil 
divisor on $\hat X$. But then $n\bar E$ is a movable divisor for some 
$n>0$. This contradicts exceptionality of $E$.
\end{proof}

If $|kA|\neq \emptyset$, let 
$S_k\in |kA|$ be a general member.
Write
\begin{equation}
\label{eq-3-fol-discr}
\begin{array}{lll}
K_{\tilde X}&=&f^*K_X+\alpha E,
\\[4pt]
\tilde S_k&=&f^*S_k-\beta_kE,
\\[4pt]
\tilde \MMM&=&f^*\MMM-\beta_0E.
\end{array}
\end{equation}
Then
\begin{equation}
\label{eq-3-ct}
c=\alpha/\beta_0.
\end{equation}

\begin{remark}
\label{remark-eq3.3}
If $\alpha<1$, then $a(E,|{-}K_X|)<1$. On the other hand, $0=K_X+|{-}K_X|$ is 
Cartier. Hence, $a(E,|{-}K_X|)\le 0$ and so $f$ is $f^{-1}_*|{-}K_X|\subset |{-}K_{\tilde X}|$. Therefore,
\begin{equation*}
\dim\lvert-K_{\bar X}\rvert=
\dim\lvert-K_{\tilde X}\rvert\ge\dim\lvert-K_X\rvert.
\end{equation*}
\end{remark}

In our situation $X$ has only cyclic quotient singularities (see Corollary \ref{corollary-cyclic}).
So, the following result is very important.
\begin{theorem}[\cite{Kawamata-1996}]
\label{theorem-Kawamata-blowup}
Let $(Y\ni P)$ be a terminal cyclic quotient singularity of type
$\frac1r(1,a,r-a)$, let $f\colon \tilde Y\to Y$ be a Mori
divisorial contraction, and let $E$ be the exceptional divisor. 
Then $f(E)=P$, $f$ is the weighted blowup with weights $(1,a,r-a)$
and the discrepancy of $E$ is $a(E)=1/r$.
\end{theorem}
We call this $f$ the \textit{Kawamata blowup} of $P$.

\begin{case}
\label{c-Cl}
\textbf{Notation.}
Assume that $g$ is birational. Let $\bar F$ be the 
$g$-exceptional divisor and let $\tilde F$
and $F$ be its proper transforms on $\tilde X$ and $X$, respectively.
Let $n$ be the maximal integer dividing the class of $\bar F$ in $\Cl(\bar X)$.
Let $\Theta$ be an ample Weil divisor 
on $\hat X$ that generates $\Cl(\hat X)/{\qq}$.
Write $\hat S_k\qq s_k\Theta$ and $\hat E\qq e \Theta$, where
$s_k,\, e\in \ZZ$, $s_k\ge 0$, $e\ge 1$.
Note that $s_k=0$ if and only if $\bar S_k$ is contracted by $g$.
\end{case}

\begin{lemma}
\label{lemma-c-Cl} 
In the above notation assume that the group $\Cl(X)$ is torsion free. Write $F\sim dA$, 
where $d\in \ZZ$, $d\ge 1$.
Then $\Cl(\hat X)\simeq \ZZ\oplus \ZZ_n$ and $d=ne$.
\end{lemma}
\begin{proof}
Write $\bar F\sim n\bar G$, where $\bar G$ is an integral Weil divisor. 
Then $\bar E\sim e\bar \Theta +k\bar G$
for some $k\in \ZZ$ and 
$\Cl(\hat X)\simeq \Cl(\bar X)/ \bar F\ZZ\simeq \ZZ\oplus \ZZ_n$.
We have 
\[
\ZZ_d\simeq\Cl(X)/ \langle F\rangle \simeq
\Cl(\bar X)/\langle\bar E, \, \bar F\rangle
\simeq \ZZ\oplus\ZZ /\langle e\bar \Theta +k\bar G,\, n\bar G\rangle.
\]
Since the last group is of order $ne$, we have $d=ne$.
\end{proof}

From now until the end of this section we consider the case where $\hat X$ is a surface.

\begin{lemma}
\label{lemma-diagram-surface}
Assume that $\hat X$ is a surface. Then 
$\hat X$ is a del Pezzo surface with Du Val singularities of type $A_n$.
The linear system $|-K_{\hat X}|$ is base point free. If moreover
the group $\Cl(X)$ is torsion free, then 
so is $\Cl(\hat X)$ and there are only the following possibilities:
\begin{enumerate}
\item 
$K_{\hat X}^2=9$, $\hat X\simeq \PP^2$;
\item 
$K_{\hat X}^2=8$, $\hat X\simeq \PP(1^2,2)$;
\item 
$K_{\hat X}^2=6$, $\hat X\simeq \PP(1,2,3)$;
\item 
$K_{\hat X}^2=5$, $\hat X$ has a unique singular point, 
point of type $A_4$.
\end{enumerate}
\end{lemma}

\begin{proof}
By the main result of \cite{Mori-Prokhorov-2008} the surface $\hat X$ has only 
Du Val singularities of type $A_n$.
Since $\rho(\hat X)=1$ and $\hat X$ is uniruled, $-K_{\hat X}$ is ample.
Further, since both~$\bar X$ and~$\hat X$ 
have only isolated singularities and $\Pic(\bar X/\hat X)\simeq\mathbb Z$,
there is a well-defined injective map $g^*\colon \Cl(\hat X) \to \Cl(\bar X)$.
Hence the group $\Cl(\hat X)$ is torsion free whenever so is $\Cl(X)$. 
The remaining part follows from the classification 
of del Pezzo surfaces with Du Val singularities (see, e.g.,
\cite{Miyanishi-Zhang-1988}).
\end{proof}

\begin{proposition}
\label{proposition-toric-surface}
In the above notation, let $\hat X$ is a surface. Let $\Gamma\in |-K_{\hat X}|$
and let $G:=g^{-1}(\Gamma)$. 
Suppose that there are two prime divisors $D_1$ and $D_2$
such that $\varphi(D_i)=Z$ and $K_{\hat X}+D_1+D_2+G\sim 0$.
Then the pair $(\bar X, D_1+D_2)$ is canonical.
If furthermore the surface
$\hat X$ is toric, then so are $\bar X$ and $X$. 
\end{proposition}
\begin{proof}
Clearly, we may replace $\Gamma$ with a general member of $|-K_{\hat X}|$. 
Note that $G$ is an elliptic ruled surface and
$K_{G}+D_1|_{G}+D_2|_{G}\sim 0$.
Hence divisors $D_1|_{G}$ and $D_2|_{G}$ are disjointed sections.
This shows that $D_1\cap D_2$ is either empty or
consists of fibres. 
Assume that $D_1\cap D_2\neq \emptyset$.
We can take $\Gamma$ so that 
$G\cap D_1\cap D_2=\emptyset$.
By adjunction $-K_{D_1}\sim \bar G|_{D_1}+D_2|_{D_1}$.
Since $D_1$ is a rational surface (birational to $\hat X$), $\bar G|_{D_1}+D_2|_{D_1}$
must be connected, a contradiction.
Thus, $D_1\cap D_2=\emptyset$. 

Therefore both divisors $D_1$ and $D_2$ contain no fibers and so
$D_1\simeq D_2\simeq \hat X$. Then 
the pair $(\bar X, D_1+D_2)$ is PLT by the Inversion of Adjunction.
Since $K_{\bar X}+D_1+D_2$ is Cartier, this pair must be canonical.
The second assertion follows by Corollary \ref{corollary-toric-Q-conic}
below.
\end{proof}

\begin{lemma}
\label{lemma-toric-surface}
Let $\varphi\colon Y\to Z$ be a $\QQ$-conic bundle
\textup(we assume that $Y$ is $\QQ$-factorial and $\rho(Y/Z)=1$\textup).
Suppose that there are two prime divisors $D_1$ and $D_2$
such that $\varphi(D_i)=Z$, the log divisor $K_Y+D_1+D_2$ is $\varphi$-linearly trivial and 
canonical.
Suppose furthermore that $Z$ is singular and let $o\in Z$ be a singular point.
Then $o\in Z$ is of type $A_{r-1}$ for some $r\ge 2$ and
there is a Sarkisov link 
\[
\xymatrix{
&\tilde Y\ar[dl]_{\sigma}\ar@{-->}[rr]^{\chi}&& \bar Y\ar[dr]^{\bar \varphi}&
\\
Y\ar[drr]^{\varphi}&&&&\bar Z\ar[dll]_{\delta}
\\
&&Z&&
} 
\]
where $\sigma$ is the Kawamata blowup of a cyclic quotient singularity
$\frac 1r(1, a, r-a)$ over $o$, $\chi$ is a sequence of flips, 
$\bar \varphi$ is a $\QQ$-conic bundle
with $\rho(\bar Y/\bar Z)=1$, and $\delta$ is a crepant contraction 
of an irreducible curve to $o$. 
Moreover, if $\bar D_i$ is the proper transform of $D_i$ on $\bar Y$,
then the divisor $K_{\bar Y}+\bar D_1+\bar D_2$ is linearly trivial over $Z$ and 
canonical.
\end{lemma}
\begin{proof}
Regard $Y/Z$ as an algebraic germ over $o$.
Since $D_i$ are generically sections, the fibration $\varphi$
has no discriminant curve. 
By \cite{Mori-Prokhorov-2006a} the central fiber $C:=\varphi^{-1}(o)_{\red}$ is 
irreducible and by the main result of \cite{Mori-Prokhorov-2008}
$Y/Z$ is toroidal, that is, it is locally analytically
isomorphic to a toric contraction. In particular,
$X$ has exactly two singular points at $C\cap D_i$
and these points are cyclic quotients of types 
$\frac 1r(1, a, r-a)$ and $\frac 1r(-1, a, r-a)$, respectively, for some $a$
with $\gcd(r,a)=1$. 

Now consider the Kawamata blowup of $C\cap D_1$. 
Let $E$ be the exceptional divisor and let $\tilde D_i$ 
be the proper transform of $D_i$.
Since $K_{\tilde Y}=\varphi^*K_Y+\frac1rE$ and
the pair $(Y,D_1+D_2)$ is canonical, we have 
\[
K_{\tilde Y}+\tilde D_1+\tilde D_2=\varphi^*(K_Y+D_1+D_2). 
\]
It is easy to check locally that the proper transform $\tilde C$ 
of the central fiber $C$ does not meet $\tilde D_1$. 
Moreover, $\tilde C\cap E$ is a smooth point of $\tilde Y$ and $E$.
Thus we have $\tilde D_1\cdot \tilde C=0$, $E\cdot \tilde C=1$, and
$\tilde D_2\cdot \tilde C=D_2\cdot C=1/r$.
Hence, $K_{\tilde Y}\cdot \tilde C=-1/r$. Since
the set-theoretical fiber over $o$ in $\tilde Y$ 
coincides with $E\cup \tilde C$, the divisor $-K_{\tilde Y}$ is ample over
$Z$ and $\tilde C$ generates a (flipping) extremal ray $R$.
Run the MMP over $Z$ in this direction, i.e., starting with $R$. 
Assume that we end up with a divisorial contraction
$\bar \varphi \colon \bar Y\to \bar Z$. Then $\bar \varphi$ must contract 
the proper transform $\bar E$ of $E$. Here $\bar Z/Z$ is a Mori conic bundle
and the map $Y\dashrightarrow \bar Z$ is an isomorphism in codimension one, 
so it is an isomorphism. 
Moreover, $\bar Z/Z$ has a section, the proper transforms of $D_i$.
Hence the fibration $\bar Z/Z$ is toroidal over $o$.
Consider Shokurov's difficulty \cite{Shokurov-1985}
\[
\diff(W):=\# \{ \text{exceptional divisors of discrepancy $<1$}\}.
\]
Then $\diff (Y)=\diff (\bar Z)=2(r-1)$. On the other hand, 
\[
\diff (\bar Z)-1 \le \diff (\bar Y)<\diff (\tilde Y)=r-1+ a-1+r-a-1=2r-3
\]
(because the map $\tilde Y\dashrightarrow \bar Y$ is not an isomorphism).
The contradiction shows that our MMP ends up with a $\QQ$-conic bundle.
Clearly, the divisor $K_{\bar Y}+\bar D_1+\bar D_2$ is linearly trivial and canonical.
By \cite{Mori-Prokhorov-2008} the surface $\bar Z$ has 
at worst Du Val singularities of type $A$.
Hence the morphism $\delta$ is crepant \cite{Morrison-1985}. 
\end{proof}

\begin{corollary}
In the above notations assume that $\bar Y$ is a toric variety.
Then so is $Y$.
\end{corollary}

\begin{corollary}
\label{corollary-toric-Q-conic}
Notation as in Lemma \xref{lemma-toric-surface}. 
Assume that the base surface $Z$ is toric.
Then so is $Y$.
\end{corollary}
\begin{proof}
Induction by the number $e$ of crepant divisors of $Z$.
If $e=0$, then $Y$ is smooth and $Y\simeq \PP(\EEE)$, where 
$\EEE$ is a decomposable rank-2 vector bundle on $Z$. 
\end{proof}

\section{Case $\qQ(X)=10$}
Consider the case $\qQ(X)=10$. 
By Proposition \ref{prop-comput} the group $\Cl(X)$ is torsion free and $\B=(7, 11)$.
For $r=7$ and $11$, let $P_r$ be a (unique) point of index $r$.
In notation of \S \ref{sect-bir-constr}, take
$\MMM:=|3A|$. Since $\dim |2A|=0$, the pencil $\MMM$ has no 
fixed components. Apply Construction \eqref{eq3.1}.
Near $P_{11}$ we have $A\sim -10K_X$, so
$\MMM \sim -8k K_X$.
By Lemma \ref{lemma-ct} we get $c\le 1/8$. 
In particular, the pair $(X,\MMM)$ is not canonical.
Take divisor $S_2\in |2A|$ and a general member $S_3\in \MMM$.
For some $a_1,\, a_2\in \ZZ$ we can write
 \begin{equation*}
\begin{array}{lllll}
K_{\tilde X}+5\tilde S_2&= &f^*(K_X+5S_2)-a_1E&\sim & -a_1E,
\\[4pt]
K_{\tilde X}+ 2\tilde S_2+2\tilde S_3&=&f^*(K_X+2S_2+2S_3)-a_2 E&\sim & -a_2 E.
\end{array}
\end{equation*}
Therefore,
\begin{equation}
\label{eq-rel-Cl-10-1}
\begin{array}{lll}
K_{\bar X}+5\bar S_2+a_1\bar E&\sim &0,
\\[4pt]
K_{\bar X}+ 2\bar S_2+2\bar S_3+a_2\bar E&\sim&0,
\end{array}
\end{equation}
where $\dim |S_2|=0$ and $\dim |S_3|=1$.
Using \eqref{eq-3-fol-discr} we obtain 
\begin{equation}
\label{eq-rel-Clb-1-10}
\begin{array}{lll}
5\beta_2&=&a_1+\alpha,
\\[4pt]
2\beta_2+2\beta_3&=&a_2+\alpha.
\end{array}
\end{equation}

Since $S_3\in \MMM$ is a general member, 
by \eqref{eq-3-ct} we have $c=\alpha/\beta_3\le 1/8$, so 
$8\alpha\le \beta_3$
and $a_2\ge 15\alpha+2\beta_2$.

\begin{case}
First we consider the case where $f(E)$ is either a curve or a 
Gorenstein point on $X$.
Then $\alpha$ and $\beta_k$ are non-negative integers.
In particular, $a_2\ge 15$.
From \eqref{eq-rel-Cl-10-1}
we obtain that $g$ is birational. Indeed, otherwise 
restricting the second relation of \eqref{eq-rel-Cl-10-1} 
to a general fiber $V$ we get that $-K_V$ is divisible by some number $a'\ge a_2\ge 15$.
This is impossible.
Thus $\hat X$ is a $\QQ$-Fano. Again from \eqref{eq-rel-Cl-10-1} we get
$\qQ(\hat X)\ge 15$. 
Moreover, $\hat E\qq \Theta$.
In particular, $|\Theta|\neq\emptyset$.
This contradicts Proposition \ref{prop-comput}.
\end{case}

\begin{case}
\label{explanations-beta-10}
Therefore $f(E)$ is a 
non-Gorenstein point $P_r$ of index $r=7$ or $11$.
By Theorem \ref{theorem-Kawamata-blowup} $\alpha=1/r$.
Near $P_r$ we can write $A\sim -l_rK_X$, where $l_r\in \ZZ$ and $10l_r\equiv 1\mod r$.
Then $S_k+kl_rK_X$ is Cartier near $P_r$. Therefore,
$\beta_k\equiv kl_r\mod \ZZ$ and we can write $\beta_k=kl_r/r+m_k$, where $m_k=m_{k,r}\in \ZZ_{\ge 0}$.
Explicitly, we have the following values of $\alpha$, $\beta_k$, and $a_k$: 
\[
\begin{array}{l|l|ll|ll}
r&\alpha&\beta_2&\beta_3&a_1&a_2
\\[4pt]
\hline
&&&&&
\\[-2pt]
7&\frac17&\frac37+m_2&\frac17+m_3&
2+5m_2&1+2m_2+2m_3
\\[10pt]
11&\frac1{11}&\frac9{11}+m_2&\frac8{11}+m_3&
4+5m_2&3+2m_2+2m_3
\end{array} 
\]

\begin{claim}
\label{claim-q10}
If $r=7$, then $m_3\ge 1$.
\end{claim}
\begin{proof}
Follows from $c=\alpha/\beta_3\le 1/8$.
\end{proof}
If $g$ is not birational, then $a_i\le 3$, so $r=7$.
By the above claim we have $a_2\ge 3$. 
In this case, $g$ is a generically $\PP^2$-bundle  and $m_2=0$
(because $-K_{\bar X}$ restricted to a general fiber is divisible by
$a_2\ge 3$).
On the other hand, $a_1=2$ and $\bar S_2$ is $g$-vertical, a contradiction.
Thus $g$ is birational. Since $\bar S_3$ is moveable, $s_3\ge 1$.
Put 
\[
u:=s_2+em_2,\qquad v:=s_3+em_3. 
 \]
\end{case}

\begin{case}
\textbf{Case: $r=11$.}
Then 
\begin{equation}
\label{eq-rel-Cls-10z-2}
\begin{array}{lllll}
\hat q&=&5s_2+(4+5m_2)e&=&5u+4e,
\\[4pt]
\hat q&=& 2s_2+2s_3+(3+2m_2+2m_3)e&=& 2u+2v+3e.
\end{array}
\end{equation}
Assume that 
$u=0$. Then $\hat q=4e$. The only solution of \eqref{eq-rel-Cls-10z-2}
is the following: $\hat q=8$, $v=1$, $e=2$. 
Hence, $s_2=0$ and $s_3=1$. 
In particular, $\dim |\Theta |\ge \dim |S_3|=1$.
On the other hand, by Lemma \ref{lemma-c-Cl} the group $\Cl(\hat X)$ is torsion free
and by Lemma \ref{lemma-7} the divisor $\Theta$ is not moveable,
a contradiction.

Therefore, 
$u\ge 1$. By the first relation in \eqref{eq-rel-Cls-10z-2}
$\hat q\ge 9$. Hence the group $\Cl(\hat X)$ is torsion free.
Then by Lemma \ref{lemma-c-Cl} we have $F\sim eA$.
Since $|A|=\emptyset$, $e\ge 2$.
Again by \eqref{eq-rel-Cls-10z-2} $\hat q\ge 13$ and $e$ is odd.
Thus, $e=3$, $u=1$, and $\hat q=17$. Further, 
$s_3+em_3=v=3$ and $s_3=3$ (because $\bar S_3$ is moveable).
By Proposition \ref{prop-comput} we have
$1=\dim |S_3|\le \dim |3\Theta|=0$, a contradiction.
\end{case}

\begin{case}
\textbf{Case: $r=7$.}
Recall that $m_3\ge 1$ by Claim \ref{claim-q10}. Write
\begin{equation}
\label{eq-rel-Cl-10z-2}
\begin{array}{lllll}
\hat q&=&5s_2+(2+5m_2)e&=&5u+2e,
\\[4pt]
\hat q&=& 2s_2+2s_3+(1+2m_2+2m_3)e&=& 2u+2v+e.
\end{array}
\end{equation}
Hence, $v=s_3+em_3\ge 1+e$.

If $u=0$, then $\hat q=2e=2v+e$, $e=2v$, and $\hat q=4v\ge 4(1+e)=4(1+2v)$,
a contradiction.
If $u=2$, then $\hat q$ is even $\ge 12$. Again we have a contradiction.

Assume that $u\ge 3$. Using the first relation in \eqref{eq-rel-Cl-10z-2}
and Proposition \ref{prop-comput} we get successively $u=3$, $\hat q\ge 17$, $|\Theta|=\emptyset$, $e\ge 2$,
$\hat q\ge 19$, $|2\Theta|=\emptyset$, $e\ge 3$, and so $\hat q\ge 21$, a contradiction.

Therefore, $u=1$. Then $\hat q=5+2e=2+2v+e$ and $2v=3+e=2v\ge 2+2e$.
So, 
$e=1$, $v=2$, $\hat q=7$. 
Since $m_3\ge 1$, $s_3=v-em_3=1$. Hence, $\hat S_3\qq \Theta$. Since $\dim |\hat S_3|\ge 1$,
by \ref{theorem-main-q7} of Theorem \ref{theorem-main} we have $\hat X\simeq \PP(1^2,2,3)$.
In particular, the group $\Cl (\hat X)$ is torsion free.
By Lemma \ref{lemma-c-Cl} the divisor $F$ generates the group 
$\Cl(X)$. This contradicts $|A|=\emptyset$.

The last contradiction finishes the  proof of \ref{theorem-main-10} of Theorem \ref{theorem-main}.
\end{case}

\section{Case $\qQ(X)=11$ and $\dim |-K_X|\ge 11$}
In this section we consider the case $\qQ(X)=11$ and $\dim |-K_X|\ge 11$. 
By Proposition \ref{prop-comput} the group $\Cl(X)$ is torsion free 
and $\B=(2,3,5)$.
For $r=2$, $3$, $5$, let $P_r$ be a (unique) point of index $r$.
In notation of \S \ref{sect-bir-constr}, take
$\MMM:=|2A|$. Since $0=\dim |A|>\dim \MMM=1$, the linear system $\MMM$ has no 
fixed components. Apply Construction \eqref{eq3.1}.
Near $P_5$ we have $A\sim -K_X$ and $\MMM\sim -2K_X$.
By Lemma \ref{lemma-ct} we get $c\le 1/2$. 
In particular, the pair $(X,\MMM)$ is not canonical.
It can be easily seen from Proposition \ref{prop-comput}
that there are reduced irreducible members $S_k\in |kA|$
for $k=1$, $2$, $3$, $5$.

\begin{proposition}
In the above notation, $f$ is the Kawamata blwup of $P_5$ 
and $\hat X$ is a del Pezzo surface with 
Du Val singularities with $K_{\hat X}^2=5$ or $6$. 
Moreover, for $k=1$, $2$ and $3$, the image 
$C_k:=g(\bar S_k)$ is a curve on $\hat X$ 
with $-K_{\hat X}\cdot C_k=k$.
\end{proposition}

\begin{proof}
Similar to \eqref{eq-rel-Cl-10-1}-\eqref{eq-rel-Clb-1-10} we have
for some $a_1,a_2,a_3\in \ZZ$:
\begin{equation}
\label{eq-rel-Cl-11}
\begin{array}{lll}
K_{\bar X}+11\bar S_1+a_1\bar E&\sim&0,
\\[4pt]
K_{\bar X}+ \bar S_1+5\bar S_2+a_2\bar E&\sim&0,
\\[4pt]
K_{\bar X}+ 2\bar S_1+3\bar S_3+a_3\bar E&\sim&0,
\end{array}
\end{equation}
\begin{equation}
\begin{array}{lll}
11\beta_1&=&a_1+\alpha,
\\[4pt]
\beta_1+5\beta_2&=&a_2+\alpha,
\\[4pt]
2\beta_1+3\beta_3&=&a_3+\alpha.
\end{array}
\end{equation}
Since $S_2\in \MMM$
is a general member, 
by \eqref{eq-3-ct} we have
$c=\alpha/\beta_2\le 1/2$, so $2\alpha\le \beta_2$
and $a_2\ge 9\alpha+\beta_1$.
Since $2S_1\sim S_2$, we have
$2\beta_1\ge \beta_2$. Thus $\beta_1\ge \alpha$ and $a_1,\, a_2\ge 10\alpha$.

First we consider the case where $f(E)$ is either a curve or a 
Gorenstein point on $X$.
Then $\alpha$ and $\beta_k$ are integers, so $a_1,\, a_2\ge 10$.
From 
\eqref{eq-rel-Cl-11}
we obtain that $g$ is birational. 
Moreover, $\hat q\ge 15$, the group 
$\Cl(\hat X)$ is torsion free, and $\hat E\sim \Theta$.
In particular, $|\Theta|\neq\emptyset$.
This contradicts Proposition \ref{prop-comput}.

\begin{case}
Therefore $P:=f(E)$ is a 
non-Gorenstein point of index $r=2,3$ or $5$. As in \ref{explanations-beta-10}
we have the following values of $\beta_k$ and $a_k$: 
{\small \[
\begin{array}{l|lll|lll}
r&\beta_1&\beta_2&\beta_3&a_1&a_2&a_3
\\[4pt]
\hline
&&&&
\\[-2pt]
2&\frac12+m_1&m_2&\frac12+m_3
&5+11m_1&
m_1+5m_2&
2+2m_1+3m_3
\\[8pt]
3&\frac23+m_1&\frac13+m_2&m_3
&7+11m_1&
1+m_1+5m_2&
1+2m_1+3m_3
\\[8pt]
5&\frac15+m_1&\frac25+m_2&\frac35+m_3
&2+11m_1&
2+m_1+5m_2&
2+2m_1+3m_3
\end{array} 
\]}
\end{case}

\begin{claim}
\label{claim-q11}
If $r=2$ or $3$, then $m_2\ge 1$.
\end{claim}
\begin{proof}
Follows from $1/2\ge c=\alpha/\beta_2=1/r\beta_2$.
\end{proof}

Assume that $g$ is birational.
By Proposition \ref{prop-comput} and Remark \ref{remark-eq3.3}
we have $\dim |-K_{\hat X}|\ge |-K_X|=23$. So, $\hat q\le 11$.
If $\bar S_1$ is not contracted, then by
the first relation in \eqref{eq-rel-Cl-11} we have
$\hat q\ge 11+a_1\ge 13$, a contradiction.
Therefore the divisor $\bar S_1$ is contracted. By Lemma \ref{lemma-c-Cl} 
the group $\Cl(\hat X)$ is torsion free and $\hat E\sim \Theta$.
Hence, $\hat q=a_1\le 7$, $m_1=0$, and $r\neq 5$.
But then $m_2\ge 1$ (see Claim \ref{claim-q11}) and $a_2\ge 5$.
This contradicts the second relation in \eqref{eq-rel-Cl-11}.

Therefore $g$ is of fiber type. 
Restricting \eqref{eq-rel-Cl-11} to a general fiber
we get $a_i\le 3$.
Thus, $r=5$ and $a_1=a_2=a_3=2$. 
Moreover, divisors $\bar S_1$, $\bar S_2$, and $\bar S_3$
are $g$-vertical.
Since $\bar S_3$ is irreducible and $\dim |\bar S_3|=2$,
$\hat X$ cannot be a curve.
Therefore $\hat X$ is a surface and
the images $g(\bar S_1)$, $g(\bar S_2)$, and $g(\bar S_3)$ are curves. 
Since $\dim |\bar S_1|=0$, we have $\dim |g(\bar S_1)|=0$.
Hence, $K_{\hat X}^2\le 6$ and $g(\bar S_1)$ is a line on 
$\hat X$. 
By Lemma \ref{lemma-diagram-surface} there are only two possibilities:
$\hat X\simeq \PP(1,2,3)$ and $\hat X$ is an $A_4$-del Pezzo surface. 
\end{proof}

\begin{case}
Consider the case where $\hat X$ is an $A_4$-del Pezzo surface. 
Assume that $\bar S_6$ is $g$-vertical.
By Riemann-Roch for Weil divisors on surfaces with Du Val singularities
\cite{Reid-YPG1987} we have 
$\dim |\bar S_6|=\dim |g(\bar S_6)|=6$.
On the other hand, $\dim |\bar S_6|=\dim |S_6|=7$, a contradiction.
Thus $g(\bar S_5)=\hat X$.
Since $K_X+S_5+S_6\sim 0$, 
\[
K_{\bar X}+\bar S_5+\bar S_6+\bar E\sim 0.
\]
Therefore $\bar S_6$ and $\bar E$ are sections of $g$.
By Proposition \ref{proposition-toric-surface}
the pair $(\bar X, \bar S_6+\bar E)$ is canonical.
Now since $\bar S_5$ is nef, the map $\bar X \dashrightarrow \tilde X$ 
is a composition of steps of the $K_{\bar X}+\bar S_6+\bar E$-MMP.
Hence the pair $(\tilde X, \tilde S_6+E)$ is also canonical.
In particular, $\tilde S_6\cap E=\emptyset$ and so
$P_5=f(E)\notin S_6$, a contradiction.
\end{case}

\begin{case}
Now consider the case $\hat X\simeq \PP(1,2,3)$. 
As above, if
$g(\bar S_5)$ is a curve, then
$\dim |g(\bar S_5)|=5$ and $g(\bar S_5)\sim 5g(\bar S_1)$.
On the other hand, $g(\bar S_5)\sim -\frac 56 K_{\hat X}$.
But then $\dim |g(\bar S_5)|=4$, a contradiction.
Therefore, $g(\bar S_5)=\hat X$.
Similar to \eqref{eq-rel-Cl-11} we have 
$K_{\bar X}+2\bar S_5+\bar S_1+a_4\bar E\sim0$.
This shows that $a_4=0$ and $\bar S_5$ is a section of $g$.
Thus we can write $K_{\bar X}+\bar S_5+G+\bar E\sim 0$,
where $G$ is a $g$-trivial Weil divisor, i.e., $G=g^*\Gamma$
for some Weil divisor $\Gamma$. Pushing down this 
equality to $X$ we get $G\sim 6\bar S_1$, i.e., $\Gamma\in |-K_{\hat X}|$. 
By 
Proposition \ref{proposition-toric-surface}
varieties $\bar X$ and $X$ are toric.
This proves \ref{theorem-main-11} of Theorem \ref{theorem-main}.
\end{case}

\section{Case $\qQ(X)=13$ and $\dim |-K_X|\ge 6$}
\label{sect-q13}
In this section we consider the case $\qQ(X)=13$ and $\dim |-K_X|\ge 6$. 
By Proposition \ref{prop-comput}\ $\B=(3,4,5)$.
For $r=3$, $4$, $5$, let $P_r$ be a (unique) point of index $r$.
In notation of \S \ref{sect-bir-constr}, take
$\MMM:=|4A|$. Since $1=\dim |3A|>\dim \MMM=2$, the linear 
system $\MMM$ has no fixed components. Apply Construction \eqref{eq3.1}.
Near $P_5$ we have $A\sim -2K_X$ and $\MMM\sim -3K_X$.
By Lemma \ref{lemma-ct} we get $c\le 1/3$. 
In particular, the pair $(X,\MMM)$ is not canonical.

\begin{proposition}
\label{proposition-q13}
In the above notation, $f$ is the Kawamata blowup of $P_5$, $g$ is birational,
it contracts $\bar S_1$,
and $\hat X\simeq \PP(1^3,2)$. Moreover, $\hat S_3\sim \hat S_4\sim \hat E\sim \Theta$ and
$\hat S_5\sim 2\Theta$. 
\end{proposition}

\begin{proof}
Similar to \eqref{eq-rel-Cl-10-1}-\eqref{eq-rel-Clb-1-10} we have
for some $a_1,a_2,a_3\in \ZZ$:
\begin{equation}
\label{eq-rel-Cl-13}
\begin{array}{lll}
K_{\bar X}+ 13\bar S_1+a_1\bar E&\sim& 0,
\\[4pt]
K_{\bar X}+\bar S_1+ 4\bar S_3+a_2\bar E&\sim& 0,
\\[4pt]
K_{\bar X}+ \bar S_1+3\bar S_4+a_3\bar E&\sim& 0,
\end{array}
\end{equation}
\begin{equation}
\label{eq-beta-13}
\begin{array}{lll}
13\beta_1&=&a_1+\alpha,
\\[4pt]
\beta_1+4\beta_3&=&a_2+\alpha,
\\[4pt]
\beta_1+3\beta_4&=&a_3+\alpha.
\end{array}
\end{equation}
Since $S_4\in \MMM$ is a general member, 
by \eqref{eq-3-ct} we have $c=\alpha/\beta_4\le 1/3$, $3\alpha\le \beta_4$
and $a_3\ge 8\alpha+\beta_1$.
Since $4S_1\sim S_4$, we have 
$4\beta_1\ge \beta_4$. Thus $\beta_1\ge \alpha$ and $a_1\ge 12\alpha$.

First we consider the case where $f(E)$ is either a curve or a 
Gorenstein point on $X$.
Then $\alpha$ and $\beta_k$ are integers.
In particular, $a_1\ge 12$.
From the first relation in
\eqref{eq-rel-Cl-13} 
we obtain that $g$ is birational. 
Moreover, $\hat q\ge 13$ and $\hat E\sim \Theta$.
In particular, $|\Theta|\neq\emptyset$.
By Proposition \ref{prop-comput} we have $\hat q=13$, $a_1=13$, $\bar S_1$ is contracted,
and $\alpha=1$. This contradicts \eqref{eq-beta-13}.

Therefore $P:=f(E)$ is a 
non-Gorenstein point of index $r=3,4$ or $5$.
By Theorem \ref{theorem-Kawamata-blowup} $\alpha=1/r$.
Similar to \ref{explanations-beta-10} we have (here $m_k\in  \ZZ_{\ge 0}$)
{
\footnotesize
\[
\begin{array}{l|llll|llllll}
r&\beta_1&\beta_3&\beta_4&\beta_5&a_1&a_2&a_3
\\[4pt]
\hline
&&&&&
\\[-2pt]
3&\frac13+m_1&m_3&\frac13+m_4&\frac23+m_5
&4+13m_1&m_1+4m_3&1+m_1+3m_4
\\[10pt]
4&\frac14+m_1&\frac34+m_3&m_4&\frac14+m_5
&3+13m_1&3+m_1+4m_3&m_1+3m_4
\\[10pt]
5&\frac25+m_1&\frac15+m_3&\frac35+m_4&m_5 
&5+13m_1&1+m_1+4m_3&2+m_1+3m_4
\end{array} 
\]}
\begin{claim}
\label{claim-q13}
If $r=3$ or $4$, then $m_4\ge 1$.
\end{claim}
\begin{proof}
Follows from $1/3\ge c=\alpha/\beta_4=1/r\beta_4$.
\end{proof}

If $g$ is not birational, then $a_1=3$, $r=4$, $m_4\ge 1$, and $a_3\ge 3$.
In this case, $a_2=a_3=3$, $g$ is a generically $\PP^2$-bundle, and 
divisors $\bar S_1$, $\bar S_3$, $\bar S_4$ are $g$-vertical.
Since $\dim |\bar S_4|>1$ and the divisor $\bar S_4$ is irreducible,
we have a contradiction.
Therefore $g$ is birational.

By Proposition \ref{prop-comput} we have $\dim |-K_{\hat X}|\ge |-K_X|=19$
and $\hat q\le 13$. From the first relation in \eqref{eq-rel-Cl-13}
we see that $\bar S_1$ is contracted. By Lemma \ref{lemma-c-Cl} the group
$\Cl(\hat X)$ is torsion free 
and $\hat E\sim \Theta$. Moreover, $m_1=0$ (because $13m_1<a_1e=\hat q\le 13$).
Thus $\hat q=a_1=4$, $3$, and $5$ in cases $r=3$, $4$, and $5$, respectively.

In cases $r=3$ and $4$ we have $\hat q\ge 3+a_3\ge 6$, a contradiction.
Therefore, $r=5$, $\hat q=5$, and $s_3=s_4=1$. Since $\dim |\Theta|\ge 1$,
by \ref{theorem-main-q7} of Theorem \ref{theorem-main} we have $\hat X\simeq \PP(1^3,2)$.
Since $\dim |S_5|=3$ and $\dim |\Theta|=2$, $s_5\ge 2$.
Similar to \eqref{eq-rel-Cl-13}-\eqref{eq-beta-13}
we have 
$K_{\bar X}+ \bar S_3+2\bar S_5+a_4\bar E\sim 0$, $2s_5+a_4=4$, and 
$a_4=\beta_3+2\beta_5-\alpha=m_3+2m_5$.
Thus, $s_5=2$ and $a_4=\beta_5=0$, i.e., $P_5\notin S_5$. 
\end{proof}

\begin{lemma}
\label{lemma-s-CAP-s-CAP-s-13}
\begin{enumerate}
\item
$S_1\cap S_3$ is a reduced irreducible curve.
\item
$S_1\cap S_3\cap S_4=\{P_5\}$.
\end{enumerate}
\end{lemma}
\begin{proof}
(i) Recall that  $A^3=1/60$ by Proposition \ref{prop-comput}.
Write $S_1\cap S_3=C+\Gamma$, where $C$ is a reduced irreducible curve 
passing through $P_5$ and $\Gamma$ is an effective $1$-cycle.
Suppose, $\Gamma\neq 0$. 
Then $1/4=S_1\cdot S_3\cdot S_5>S_5\cdot C$.
Since $P_5\notin S_5$, $C\not \subset S_5$ and $S_5\cdot C\ge 1/4$,
a contradiction. 
Hence, $S_1\cap S_3=C$. 

(ii) Assume that $S_1\cap S_3\cap S_4 \ni P\neq P_5$.
Since $1/5=S_1\cdot S_3\cdot S_4=S_4\cdot C$ and $P,\, P_5\in S_4\cap C$, we have 
$C\subset S_4$. 
If there is a component $C'\neq C$ of $S_1\cap S_4$ not contained in $S_5$, then, as above, 
$1/3=S_1\cdot S_4\cdot S_5\ge S_5\cdot C+S_5\cdot C'\ge 1/2$, a contradiction.
Thus we can write $S_1\cap S_4=C+\Gamma$, where 
$\Gamma$ is an effective $1$-cycle with $\Supp \Gamma \subset S_5$.
In particular, $P_5\notin \Gamma$.
The divisor $12A$ is Cartier at $P_3$ and $P_4$.
We get
\[
\frac15 =12A^3=12A\cdot S_1\cdot (S_4-S_3) =12A\cdot \Gamma  \in \ZZ,
\]
a contradiction.
\end{proof}

\begin{lemma}
\label{lemma-LC-1}
Let $X$ be a $\QQ$-Fano threefold and $D=D_1+\cdots +D_4$ be a 
divisor on $X$, where $D_i$ are irreducible components.
Let $P\in X$ be a cyclic quotient singularity of index $r$.
Assume that $K_X+D\qq 0$, $P\notin D_4$,
$D_1\cap D_2\cap D_3=\{P\}$, and $D_1\cdot D_2\cdot D_3=1/r$. 
Then the pair $(X,D)$ is LC.
\end{lemma}

\begin{proof}
Let $\pi\colon (X^\sharp, P^\sharp)\to (X,P)$ be 
the index-one cover. For $k=1,2,3$, 
let $D_k^\sharp$ be the preimage of $D_k$ and let $D^\sharp:=D_1^\sharp+D_2^\sharp+D_3^\sharp$.
By our assumptions $D_1^\sharp\cap D_2^\sharp\cap D_3^\sharp=\{P^\sharp\}$.
Since $D_1\cdot D_2\cdot D_3=1/r$,
locally near $P^\sharp$ we have $D_1^\sharp\cdot D_2^\sharp\cdot D_3^\sharp=1$.
Hence $D^\sharp$
is a simple normal crossing divisor  \textup(near $P^\sharp$\textup).
In particular, $(X^\sharp, D^\sharp)$ is LC near $P^\sharp$
and so is $(X, D)$ near $P$.

Thus the pair $(X,D)$ is LC in some neighborhood $U\ni P$.
Since $D_1\cap D_2\cap D_3=\{P\}$, $P$ is a center of LC singularities for $(X,D)$.
Let $H$ be a general hyperplane section through 
$P$. Write $\lambda D_4\qq H$, where $\lambda>0$.
If $(X,D)$ is not LC in $X\setminus U$, then 
the locus of log canonical singularities of the pair 
$(X,D+\epsilon H-(\lambda \epsilon +\delta)D_4)$ is 
not connected for $0<\delta\ll \epsilon \ll 1$.
This contradicts Connectedness Lemma \cite{Shokurov-1992-re}, \cite{Utah}.
Therefore the pair $(X,D)$ is LC.
\end{proof}

\begin{case}
\label{proof-theorem-13}
\textbf{Proof of \ref{theorem-main-13} of Theorem \ref{theorem-main}.}
By Lemma \ref{lemma-LC-1} 
the pair $(X, S_1+S_3+S_4+S_5)$ is LC.
Since $K_X+S_1+S_3+S_4+S_5\sim 0$, it is easy to see that $a(E,S_1+S_3+S_4+S_5)=-1$.
Thus $K_{\tilde X}+\tilde S_1+\tilde S_3+\tilde S_4+\tilde S_5=f^*(K_X+S_1+S_3+S_4+S_5)\sim 0$.
Therefore the pairs $(\bar X, \bar S_1+\bar S_3+\bar S_4+\bar S_5+\bar E)$ 
and $(\hat X, \hat S_3+\hat S_4+\hat S_5+\hat E)$ are also LC.
It follows from  Proposition \ref{proposition-q13} and its proof that 
$\hat X\simeq \PP(1^3,2)$,
$\hat E\sim \hat S_3\sim \hat S_4\sim \Theta$, and $\hat S_5\sim 2\Theta$.
We claim that $\hat S_3+\hat S_4+\hat S_5+\hat E$ 
is a toric boundary (for a suitable choice of coordinates in $\PP(1^3,2)$). 
Let $(x_1:x_1': x_1'': x_2)$ be homogeneous coordinates in $\PP(1^3,2)$.
Clearly, we may assume that $\hat E=\{x_1=0\}$, $\hat S_3=\{x_1'=0\}$, and 
$\hat S_4=\{\alpha x_1+\alpha' x_1'+\alpha'' x_1''=0\}$ for some constants
$\alpha$, $\alpha'$, $\alpha''$. Since
$(\hat X, \hat S_3+\hat S_4+\hat E)$ is LC, $\alpha''\neq 0$ and
after a coordinate change we may assume that $\hat S_4=\{x_1''=0\}$.
Further, the surface $\hat S_5$ is given by the equation $\beta x_2+\psi(x_1,x_1', x_1'')=0$,
where $\beta$ is a constant and $\psi$ is a quadratic form.
If $\beta =0$, then $\hat S_3\cap \hat S_4\cap \hat E\cap \hat S_5\neq \emptyset$
and the pair $(\hat X, \hat S_3+\hat S_4+\hat S_5+\hat E)$ cannot be  LC.
Thus $\beta \neq 0$ and after a coordinate change we may assume that
$\hat S_5=\{x_2=0\}$. Therefore $\hat S_3+\hat S_4+\hat S_5+\hat E$ 
is a toric boundary. 
Then by Lemma \ref{lemma-birational-toric-pairs-mck} below the varieties $\bar X$,
$\tilde X$, and $X$ are toric.
This proves \ref{theorem-main-13} of Theorem \ref{theorem-main}.
\end{case}

\begin{lemma}[see, e.g., {\cite[3.4]{McKernan-2001}}]
\label{lemma-birational-toric-pairs-mck}
Let $V$ be a toric variety and let $\Delta$ be the toric \textup(reduced\textup) boundary.
Then every valuation $\nu$ 
with discrepancy $-1$ with respect to $K_V+\Delta$ is toric, that is,
there is a birational toric morphism $\tilde V\to V$ such that 
$\nu$ corresponds to an exceptional divisor.
\end{lemma}

\section{Case $\qQ(X)=17$}
Consider the case $\qQ(X)=17$. 
By Proposition \ref{prop-comput}\ $\B=(2,3,5,7)$.
For $r=2$, $3$, $5$, $7$, let $P_r$ be a (unique) point of index $r$.
In notation of \S \ref{sect-bir-constr}, take
$\MMM:=|5A|$ and apply Construction \eqref{eq3.1}. 
Near $P_7$ we have $A\sim -5K_X$ and $\MMM\sim -4K_X$.
By Lemma \ref{lemma-ct} we get $c\le 1/4$. 
In particular, the pair $(X,\MMM)$ is not canonical.

\begin{proposition}
In the above notation, $f$ is the Kawamata blowup of $P_7$, $g$ is birational,
it contracts $\bar S_2$,
and $\hat X\simeq \PP(1^2,2,3)$. Moreover, $\hat S_3\sim \hat S_5\sim \Theta$, $\hat E\sim 2\Theta$,
and $\hat S_7\sim 3\Theta$.
\end{proposition}

\begin{proof}
Similar to \eqref{eq-rel-Cl-10-1}-\eqref{eq-rel-Clb-1-10} we have
for some $a_1,a_2,a_3\in \ZZ$:
\begin{equation}
\label{eq-rel-Cl-17}
\begin{array}{lll}
K_{\bar X}+ 7\bar S_2+\bar S_3+a_1\bar E&\sim& 0,
\\[4pt]
K_{\bar X}+ \bar S_2+5\bar S_3+a_2\bar E&\sim &0,
\\[4pt]
K_{\bar X}+ \bar S_2+3\bar S_5+a_3\bar E&\sim& 0,
\end{array}
\end{equation}
\begin{equation}
\label{eq-beta-17}
\begin{array}{lll}
7\beta_2+\beta_3&=&a_1+\alpha,
\\[4pt]
\beta_2+5\beta_3&=&a_2+\alpha,
\\[4pt]
\beta_2+3\beta_5&=&a_3+\alpha.
\end{array}
\end{equation}
Since $S_5\in \MMM$ 
is a general member, 
by \eqref{eq-3-ct} we have
$c=\alpha/\beta_5\le 1/4$, so $4\alpha\le \beta_5$
and $a_3\ge 11\alpha+\beta_2$.
Since $S_2+S_3\sim S_5$, we have 
$\beta_2+\beta_3\ge \beta_5\ge 4\alpha$. 
Hence, $a_1\ge 6\beta_2+3\alpha$ and $a_2\ge 4\beta_3+3\alpha$.

First we consider the case where $f(E)$ is either a curve or a 
Gorenstein point on $X$.
Then $\alpha$ and $\beta_k$ are integers.
In particular, $a_3\ge 11$ and by the third relation in \eqref{eq-rel-Cl-17} 
we obtain that $g$ is birational. 
Moreover, $\hat q\ge 11$.
In particular, the group $\Cl(\hat X)$ is torsion free and so
$\hat E\ge 2\Theta$.
Hence, $\hat q\ge 2a_3\ge 22$, a contradiction.

Therefore $P:=f(E)$ is a 
non-Gorenstein point of index $r=2,3, 5$ or $7$.
Similar to \ref{explanations-beta-10} we have $\alpha=1/r$ and
{\scriptsize
\[
\begin{array}{l|llll|lll}
r&\beta_2&\beta_3&\beta_5&\beta_7&a_1&a_2&a_3
\\
\hline
&&&&
\\
2&m_2&\frac12+m_3&\frac12+m_5&\frac12+m_7&7m_2+m_3&2+m_2+5m_3&1+m_2+3m_5
\\[10pt]
3&\frac13+m_2&m_3&\frac13+m_5&\frac23+m_7&2+7m_2+m_3&m_2+5m_3&1+m_2+3m_5
\\[10pt]
5&\frac15+m_2&\frac45+m_3&m_5&\frac15+m_7&2+7m_2+m_3&4+m_2+5m_3&m_2+3m_5
\\[10pt]
7&\frac37+m_2&\frac17+m_3&\frac47+m_5&m_7&3+7m_2+m_3&1+m_2+5m_3&2+m_2+3m_5
\end{array} 
\]
}
\begin{claim}
\label{claim-q17}
\begin{enumerate}
\item 
If $r=2$, then $m_5\ge 2$ and $m_2+m_3\ge 2$. 
\item 
If $r=3$, then $m_5\ge 1$ and $m_2+m_3\ge 1$.
\item 
If $r=5$, then $m_5\ge 1$.
\end{enumerate}
\end{claim}
\begin{proof}
Note that $1/4\ge c=\alpha/\beta_5=1/r\beta_5$ and $r\beta_5\ge 4$.
This gives us inequalities for $m_5$.
The inequalities for $m_2+m_3$ follows from $\beta_2+\beta_3\ge \beta_5$.
\end{proof}

From this we have $\min (a_1,a_2, a_3)\ge 3$.
Moreover, the equality $\min (a_1,a_2, a_3)=3$ holds only if $r=7$.
Therefore the contraction
$g$ can be of fiber type only if $a_1=3$, $r=7$, $m_2=m_3=0$,
$\min (a_1,a_2, a_3)=3$,
$r=7$, $m_2=m_3=m_5=0$, $a_3=2$, and $a_2=1$. 
Then $g$ is a del Pezzo fibration of degree
$9$ and by the first relation in \eqref{eq-rel-Cl-17} divisors $\hat S_2$ and $\hat S_3$
are $g$-vertical. But then $a_2=3$, a contradiction. 
From now on we assume that $g$ is birational.

Since $\bar S_5$ is moveable, it is not contracted.
Therefore, $s_5\ge 1$. By \eqref {eq-rel-Cl-17} we have
\begin{equation*}
\begin{array}{lll}
\hat q&=& 7s_2+s_3+a_1e, 
\\[4pt]
\hat q&=&s_2+5s_3+a_2e, 
\\[4pt]
\hat q&=&s_2+3s_5+a_3e.
\end{array}
\end{equation*}
Put 
\[
u:=s_2+em_2,\qquad v:=s_3+em_3,\qquad w:=s_5+em_5. 
 \]

\begin{case} \textbf{Case: $r=2$.} 
Then $a_3\ge 7$ and $\hat q\ge 3s_5+a_3\ge 10$.
Hence the group $\Cl(\hat X)$ is torsion free. So, $e\ge 2$ and
$\hat q\ge 3s_5+2a_3\ge 17$. In this case $|\Theta|=\emptyset$.
Therefore,
$s_5\ge 2$ and $\hat q\ge 3s_5+2a_3\ge 20$, a contradiction.
\end{case}

\begin{case} \textbf{Case: $r=3$.} 
Then 
\[
\begin{array}{llrll}
\hat q & =& 7s_2+s_3+(2+7m_2+m_3)e&=& 7u+v+2e,
\\[4pt]
\hat q & =& s_2+5s_3+(m_2+5m_3)e&=&u+5v, 
\\[4pt]
\hat q & =& s_2+3s_5+(1+m_2+3m_5)e&=& u+3w+ e.
\end{array}
\]
Assume that $u>0$. Then $\hat q\ge 9$. Hence the group $\Cl(\hat X)$ is torsion free and $e\ge 2$.
Since $\dim |S_5|=1$ and $\dim |\Theta|\le 0$, we have $s_5\ge 2$. 
Since $m_5\ge 1$ (see Claim \ref{claim-q17}), we have $w\ge 4$ and $\hat q>13$.
In this case, $s_5\ge 5$, a contradiction.

Therefore, $u=0$, $m_2=0$, $s_3\neq 0$, $m_3\ge 1$, and
$v\ge 2$. So, 
$\hat q =5v\ge 10$.
Then we get a contradiction by \ref{theorem-main-10} of Theorem \ref{theorem-main}.
\end{case}

\begin{case} \textbf{Case: $r=5$.} 
Then 
\[
\begin{array}{lllll}
\hat q & =& 7s_2+s_3+(2+7m_2+m_3)e& =&7u+v+2e,
\\[4pt]
 \hat q & =& s_2+5s_3+(4+m_2+5m_3)e& =&u+5v+4e, 
\\[4pt]
\hat q & =& s_2+3s_5+(m_2+3m_5)e& =&u+3w.
\end{array}
\]
From the first two relations we have 
$3u=2v+e$ and $1\le u\le 2$. 
Further, $\hat q-4u=3(v+e)$, so $\hat q \equiv u\mod 3$.

If $u=2$, then 
$e$ is even and $\hat q=14+v+2e\ge 18$. So, $\hat q=19$,
a contradiction.

Thus $u=1$, $3=2v+e$, and $\hat q=7+v+2e\ge 9$. 
By \ref{theorem-main-10} of Theorem \ref{theorem-main} $\hat q$ is odd.
Hence, $v$ is even, $e=3$, $v=0$, $\hat q=13$. In this case,
$s_5+3m_5=w=4$. By Claim \ref{claim-q17} $m_5=s_5=1$.
Note that the group $\Cl(\hat X)$ is torsion free
and $s_2=1$. Thus $\dim |\Theta|>0$.
This contradicts Proposition \ref{prop-comput}.
\end{case}

\begin{case} \textbf{Case: $r=7$.} 
Then 
\[
\begin{array}{lllll}
\hat q & =& 7s_2+s_3+(3+7m_2+m_3)e& =& 7u+v+3e,
\\[4pt]
 \hat q & =& s_2+5s_3+(1+m_2+5m_3)e& =&u+5v+e, 
\\[4pt]
\hat q & =& s_2+3s_5+(2+m_2+3m_5)e& =&u+3w+2e.
\end{array}
\]
Assume that 
$u>0$. Then $\hat q\ge 10$, the group $\Cl(\hat X)$ is torsion free
and so $e\ge 2$, $\hat q\ge 13$, $u=1$.
From the first two relations we get $\hat q+2=7v$.
Hence, $v=3$, $\hat q=19$, $e=3$, and $s_2=0$. This contradicts
the equality $1=u=s_2+em_2$.

Therefore, $u=0$ and $s_2=m_2=0$.
From the first two relations we get $\hat q=7v$.
Thus, $\hat q=7$, $v=1$, $e=2$, $w=1$, $m_3=m_5=0$, and
$s_3=s_5=1$.
By Lemma \ref{lemma-c-Cl} the group $\Cl(\hat X)$ is torsion free
and so $\dim |\Theta|\ge \dim |\bar S_5|>0$.
From \ref{theorem-main-q7} of Theorem \ref{theorem-main} we have 
$\hat X\simeq \PP(1^2,2,3)$. In particular, $\dim |\Theta|=1$. 
Further, similar to \eqref{eq-rel-Cl-17} we have
\begin{equation*}
\label{eq-rel-Cl-17-4-pp}
\begin{array}{l}
K_{\bar X}+ \bar S_3+2\bar S_7+a_4\bar E\sim 0,
\\[4pt]
\beta_3+2\beta_7=a_4+\alpha.
\end{array}
\end{equation*}
This gives us
$a_4=2\beta_7$ and
$s_7+a_4=3$. 
Since $\dim |S_7|=2$, $s_7>1$,
$s_7=3$, $\hat S_7\sim 3\Theta$, $a_4=0$, and $\beta_7=0$, i.e., $P_7\notin S_7$.
\end{case}
\end{proof}

\begin{lemma}
\label{lemma-s-CAP-s-CAP-s-17}
\begin{enumerate}
\item
$S_2\cap S_3$ is a reduced irreducible curve.
\item
$S_2\cap S_3\cap S_5=\{P_7\}$.
\end{enumerate}
\end{lemma}
\begin{proof}
(i) 
Similar to the proof of (i) of Lemma \ref{lemma-s-CAP-s-CAP-s-13}.

(ii) Put $C:=S_3\cap S_4$. Assume that $S_2\cap S_3\cap S_5\ni P\neq P_7$.
Since $1/7=S_2\cdot S_3\cdot S_5=S_5\cdot C$ and $P,\, P_7\in S_5\cap C$, we have 
$C\subset S_5$. 
If there is a component $C'\neq C$ of $S_2\cap S_5$ not contained in $S_7$, then, as above, 
$7/15=S_2\cdot S_7\cdot S_7\ge S_7\cdot C+S_7\cdot C'\ge 2/5$, a contradiction.
Thus we can write $S_2\cap S_5=C+\Gamma$, where 
$\Gamma$ is an effective $1$-cycle with $\Supp \Gamma \subset S_7$.
In particular, $P_7\notin \Gamma$.
The divisor $30A$ is Cartier at $P_2$, $P_3$, and $P_5$.
We get
\[
\frac{120}{210}= 120A^3  = 30A\cdot S_2\cdot (S_5-S_3)= 30A\cdot \Gamma\in \ZZ,
\]
a contradiction.
\end{proof}
Now the proof of \ref{theorem-main-17} of Theorem \ref{theorem-main}
can be finished similar to \ref{proof-theorem-13}:
the pair $(\hat X, \hat S_3+\hat S_5+\hat E+\hat S_7)$ is LC and
the corresponding discrepancy of $\bar S_2$ is equal to $-1$.

\section{Case $\qQ(X)=19$}
Consider the case $\qQ(X)=19$. By Proposition \ref{prop-comput}\ $\B=(3,4,5,7)$.
For $r=3$, $4$, $5$, $7$, let $P_r$ be a (unique) point of index $r$.
In notation of \S \ref{sect-bir-constr}, take
$\MMM:=|7A|=|S_7|$ and apply Construction \eqref{eq3.1}.
Near $P_5$ we have $A\sim -4K_X$ and $\MMM\sim -3K_X$.
By Lemma \ref{lemma-ct} we get $c\le 1/3$. 
In particular, the pair $(X,\MMM)$ is not canonical.

\begin{proposition}
\label{proposition-q19}
In the above notation, $f$ is the Kawamata blowup of $P_5$, $g$ is birational,
it contracts $\bar S_3$,
and $\hat X\simeq \PP(1^2,2,3)$.
Moreover, $\hat S_4\sim \hat S_7\sim \Theta$, $\hat E\sim 3\Theta$,
and $\hat S_5\sim 2\Theta$.
\end{proposition}

\begin{proof}
Similar to \eqref{eq-rel-Cl-10-1}-\eqref{eq-rel-Clb-1-10} we have
for some $a_1,a_2,a_3, a_4\in \ZZ$:
\begin{equation}
\label{eq-rel-Cl-19}
\begin{array}{lll}
K_{\bar X}+ 5\bar S_3+\bar S_4+a_1\bar E&\sim& 0,
\\[4pt]
K_{\bar X}+ \bar S_3+4\bar S_4+a_2\bar E&\sim& 0,
\\[4pt]
K_{\bar X}+ \bar S_4+3\bar S_5+a_3\bar E&\sim &0,
\\[4pt]
K_{\bar X}+ \bar S_5+2\bar S_7+a_4\bar E&\sim &0,
\end{array}
\end{equation}
\begin{equation}
\label{eq-rel-Cl-19-u}
\begin{array}{lll}
5\beta_3+\beta_4&=&a_1+\alpha,
\\[4pt]
\beta_3+4\beta_4&=&a_2+\alpha,
\\[4pt]
\beta_4+3\beta_5&=&a_3+\alpha,
\\[4pt]
\beta_5+2\beta_7&=&a_4+\alpha.
\end{array}
\end{equation}
\begin{remark}
\label{remark-first-19}
Since $S_7\in \MMM$ is a general member, 
by \eqref{eq-3-ct} we have
$c=\alpha/\beta_7\le 1/3$, so $3\alpha\le \beta_7$
and $a_4\ge 5\alpha+\beta_5$.
Further, $S_3+S_4\sim S_7$. Thus, 
$\beta_3+\beta_4\ge \beta_7\ge 3\alpha$, 
$a_1\ge 4\beta_3+2\alpha$, and $a_2\ge 3\beta_4+2\alpha$.
\end{remark}

Assume that $\hat X$ is a surface. 
Then $\hat X$ is such as in Lemma \ref{lemma-diagram-surface}.
From the first and second relations in
\eqref{eq-rel-Cl-19}
we obtain that $S_3$ and $S_4$ are $g$-vertical. 
Since 
$\dim |\bar S_k|=0$, $\dim |g(\bar S_k)|=0$, $k=3$, $4$.
Hence, $K_{\hat X}^2\le 6$ and the curves $g(\bar S_k)$ are in fact lines on 
$\hat X$. In particular, $g(\bar S_3)\sim g(\bar S_4)$.
This implies $\bar S_3\sim \bar S_4$ and 
$S_3\sim S_4$, a contradiction.

Now assume that $\hat X$ is a curve and let $G$
be a general fiber of $g$. 
Clearly, divisors $\bar S_3$ and $\bar S_4$ 
are $g$-vertical.
If the divisor $\bar S_5$ is also
$g$-vertical, then
$k_3 \bar S_3\sim k_4 \bar S_4\sim k_5 \bar S_5\sim G$, where 
the $k_i$ are the multiplicities of corresponding fibres.
Considering proper transforms on $X$ we get 
$3k_3=4k_4=5k_5$ and so
$k_3=20k$, $k_4=14k$, $k_5=12k$ for some $k\in \ZZ$.
This contradicts the main result of \cite{Mori-Prokhorov-2008d}.
Therefore the divisor $\bar S_5$ is $g$-horizontal.
In this case the degree of the general fiber is $9$.
As above we have $k_3 \bar S_3\sim k_4 \bar S_4\sim G$,
$3k_3=4k_4$. So, $k_3=4k$, $k_4=3k$, $k\in \ZZ$.
Again by \cite{Mori-Prokhorov-2008d} $g$ has no 
fibers of multiplicity divisible by $4$.

From now on we assume that $g$ is birational.
Then 
\begin{equation}
\label{equation-qse-19}
\hat q=5s_3+s_4+a_1e=s_3+4s_4+a_2e=s_4+3s_5+a_3e.
\end{equation}

Consider the case where $f(E)$ is either a curve or a 
Gorenstein point on $X$.
Then $\alpha$ and $\beta_k$ are integers.
By Remark \ref{remark-first-19}
\[
a_1+a_2=5(\beta_3+\beta_4)+\beta_3-2\alpha\ge 13\alpha\ge 13.
\]
On the other hand, from
\eqref{equation-qse-19}
we obtain
$2\hat q\ge 6s_3+5s_4+13\ge 18$. So, $\hat q\ge 9$
(both $\bar S_3$ and $\bar S_4$ cannot be contracted).
In this case, the group $\Cl(\hat X)$ is torsion free and by Lemma \ref{lemma-c-Cl} we have 
$\hat E\ge 3\Theta$. Since $a_4\ge 5$, we have
$\hat E\sim 3\Theta$, $\hat q\ge 15$, and $\bar S_3$ is contracted. 
In this situation, $|\Theta|=\emptyset$, so $s_5,\, s_7\ge 2$.
This contradicts the fourth relation in \eqref{eq-rel-Cl-19}.

Therefore  $P:=f(E)$ is a 
non-Gorenstein point of index $r=3,4, 5$ or $7$.
Similar to \ref{explanations-beta-10} we have  $\alpha=1/r$ and
{\scriptsize
\[
\begin{array}{l|llll|lll}
 r& \beta_3& \beta_4
& \beta_5& \beta_7& a_1& a_2& a_3
\\
\hline
&&&&&&&
\\ 
 3& m_3&\frac13+m_4&\frac23+m_5&\frac13+m_7& 5m_3+m_4& 1+m_3+4m_4& 2+m_4+3m_5
\\[10pt]
4&\frac14+m_3&m_4&\frac34+m_5&\frac14+m_7&1+5m_3+m_4&m_3+4m_4&2+m_4+3m_5
\\[10pt]
5&\frac25+m_3&\frac15+m_4&m_5&\frac35+m_7&2+5m_3+m_4&1+m_3+4m_4&m_4+3m_5
\\[10pt]
7&\frac27+m_3&\frac57+m_4&\frac17+m_5&m_7&2+5m_3+m_4&3+m_3+4m_4&1+m_4+3m_5
\end{array} 
\]
}
\begin{claim}
\label{claim-q19}
\begin{enumerate}
\item 
If $r=3$ or $4$, then $m_7\ge 1$ and $m_3+m_4\ge 1$. 
\item 
If $r=7$, then $m_7\ge 1$.
\end{enumerate}
\end{claim}
\begin{proof}
To get inequalities for $m_7$ we use $1/3\ge c=\alpha/\beta_7=1/r\beta_7$, $r\beta_7\ge 3$.
The inequalities for $m_3+m_4$ follows from $\beta_3+\beta_4\ge \beta_7$.
\end{proof}

Thus, in all cases $a_1, a_2\ge 1$.
Put 
\[
u:=s_3+em_3,\qquad v:=s_4+em_4,\qquad w:=s_5+em_5. 
 \]

\begin{case} \textbf{Case: $r=3$.}
Then $u+v>e(m_3+m_4)\ge e$ by Claim \ref{claim-q19}. Further, 
\begin{equation*}
\begin{array}{lllll}
\hat q&=& 5s_3+s_4+(5m_3+m_4)e&=& 5u+ v,
\\[4pt]
\hat q&=& s_3+4s_4+(1+m_3+4m_4)e&=& u+4v+e,
\\[4pt]
\hat q&=& s_4+3s_5+(2+m_4+3m_5)e&=& v+3w+2e.
\end{array}
\end{equation*}

If $u=0$, then $v=\hat q=e+4 v$, a contradiction.

Assume that $u\ge 2$. Then $\hat q\ge 10$, $u\le 3$, the group $\Cl(\hat X)$ is torsion free
and by Lemma \ref{lemma-c-Cl} we have $e\ge 3$.
If $u=2$, then $v\ge 2$, $\hat q\ge 13$, $v=\hat q-10$, and $e\le \hat q-2-4v\le 2$,
a contradiction.
If $u=3$, then $v=2$, $e=6$, $\hat q=17$, and $m_3=m_4=0$. This contradicts Claim \ref{claim-q19}.

Therefore, $u=1$. Then $v=\hat q-5$, $19=e+3\hat q$, and
$\hat q\le 6$.
We get only one solution: 
$\hat q=6$, $u=v=w=e=1$.
Recall that $m_3+m_4\ge 1$ by Claim \ref{claim-q19}. 
Hence either $s_3=0$ and $\hat S_4\qq \hat S_5\qq \hat E\qq \Theta$
or $s_4=0$ and $\hat S_3\qq \hat S_5\qq \hat E\qq \Theta$.
In both cases $\hat S_5\not\sim \hat E$
(otherwise $\bar S_5\sim \bar E+l\bar F$ for some 
$l\in \ZZ$ and so $S_5\sim l F$, a contradiction).
Then we get a contradiction by Lemma \ref{proposition-main}.
\end{case}

\begin{case} \textbf{Case: $r=4$.}
As in the previous case, $u+v> e$ and
\begin{equation*}
\begin{array}{lllll}
\hat q&=& 5s_3+s_4+(1+5m_3+m_4)e&=&5u+v+e,
\\[4pt]
\hat q&=& s_3+4s_4+(m_3+4m_4)e&=&u+4v.
\end{array}
\end{equation*}
If $u$ is even, then so is $\hat q$. Hence, $\hat q\le 10$. From the first relation we 
have $u=0$, $\hat q=4v$, and $e=3v$.
This contradicts $u+v>e$.
Therefore $u$ is odd.

Assume that $u=1$. Then $\hat q=5+v+e=1+4v$ and $e=3v-4$.
Since $u+v>e$, there is only one possibility: 
$v=e=2$, $\hat q=9$.
Then the group $\Cl(\hat X)$ is torsion free. By Lemma \ref{lemma-c-Cl} we have
$F\in |2A|\neq \emptyset$, a contradiction.

Finally, assume $u\ge 3$. Then $u=3$ and
$\hat q=15+v+e=3+4v\ge 16$. 
Thus, $\hat q=19$, $v=4$, and $e=0$, a contradiction.
\end{case}

\begin{case} \textbf{Case: $r=7$.}
 Then
\begin{equation*}
\begin{array}{lllll}
\hat q&=& 5s_3+s_4+(2+5m_3+m_4)e&=& 5u+v+2e,
\\[4pt]
\hat q&=& s_3+4s_4+(3+m_3+4m_4)e&=& u+4v+3e,
\\[4pt]
\hat q&=& s_4+3s_5+(1+m_4+3m_5)e&=& v+3w+e.
\end{array}
\end{equation*}

In this case $u=(3v+e)/4>0$. Assume that $u\ge 2$. Then $\hat q\ge 13$ and 
the group $\Cl(\hat X)$ is torsion free. By Lemma \ref{lemma-c-Cl} we have $e\ge 3$.
Further, $u=2$, and $\hat q\ge 17$.
We get $m_3=0$, $s_3=2$, $e\ge 4$, $\hat q=19$, $e=4$, and $v=1$.
This contradicts the last relation.

Therefore, $u=1$. Then $3v+e=4$. 
Assume that $e=4$. Then $v=0$, $\hat q=13$, $w=3$, $s_4=0$, $s_3=1$, 
and $m_4=m_3=0$.
Since $\dim |\Theta|=\dim |2\Theta|=0$, we have 
$s_5\ge 3$. Recall that $m_7\ge 1$ by Claim \ref{claim-q19}. 
Hence, $\beta_7\ge 1$ and $a_4=2\beta_7\ge 2$.
This contradicts the fourth relation in \eqref{eq-rel-Cl-19}.

Therefore, $e<4$. In this case, $e=1$, $v=1$, and $\hat q=8$. 
Then $\hat E\qq \Theta$ and either $\hat S_3\qq \Theta$ or 
$\hat S_4\qq \Theta$ (because $u=v=1$).
This contradicts \ref{theorem-main-q7} of Theorem \ref{theorem-main}.
\end{case}

\begin{case} \textbf{Case: $r=5$.}
From \eqref{eq-rel-Cl-19} we obtain
\begin{equation}
\label{eq-19-f}
\begin{array}{lllll}
\hat q&=& 5s_3+s_4+(2+5m_3+m_4)e&=&5u+v+2e,
\\[4pt]
\hat q&=& s_3+4s_4+(1+m_3+4m_4)e&=&u+4v+e,
\\[4pt]
\hat q&=& s_4+3s_5+(m_4+3m_5)e&=&v+3w.
\end{array}
\end{equation}
Then $e=3v-4u$. If $u\ge 2$, then $e=3v-4u\le 3v-6$, and so $v\ge 3$.
Hence, $\hat q\ge 15$ and the group $\Cl(\hat X)$ is torsion free.
By Lemma \ref{lemma-c-Cl} we have $e\ge 3$. So $\hat q=19$, 
$e=3$, $s_3=0$, and $2=u=em_3\ge 3$, a contradiction.

Assume that $u=1$, then $e=3v-4$ and $v\ge 2$. 
Further, $\hat q=7v-3=v+3w\le 19$. We get
$\hat q =11$ and $e=2$. This contradicts Lemma \ref{lemma-c-Cl}.

Therefore, $u=0$. Then $e=3v$, $\hat q=7v=7$,
$v=1$, $e=3$, and $w=2$. By Lemma \ref{lemma-c-Cl}
the group $\Cl(\hat X)$ is torsion free.
Thus $s_3=0$, i.e., $\bar S_3$ is contracted,
$s_4=1$, $s_5=2$, and $m_5=\beta_5=0$.
This means, in particular, that $P_5\notin S_5$.
From the fourth relation in \eqref{eq-rel-Cl-19} we get $a_4=1$ and $s_7=1$.
In particular, $\dim |\Theta|>0$ and $\hat X\simeq \PP(1^2,2,3)$ by 
\ref{theorem-main-q7} of Theorem \ref{theorem-main}.
\end{case}
\end{proof}

\begin{lemma}
\label{lemma-s-CAP-s-CAP-s-19}
\begin{enumerate}
\item
$S_3\cap S_4$ is a reduced irreducible curve.
\item
$S_3\cap S_4\cap S_7=\{P_5\}$.
\end{enumerate}
\end{lemma}
\begin{proof}
(i) 
Similar to the proof of (i) of Lemma \ref{lemma-s-CAP-s-CAP-s-13}.

(ii) Put $C:=S_3\cap S_4$. 
Assume that $S_3\cap S_4\cap S_7\ni P\neq P_5$.
Since $1/5=S_3\cdot S_4\cdot S_7=S_7\cdot C$ and $P,\, P_5\in S_7\cap C$, we have 
$C\subset S_7$. 
If there is a component $C'\neq C$ of $S_3\cap S_7$ not contained in $S_5$, then, as above, 
$1/4=S_3\cdot S_7\cdot S_5\ge S_5\cdot C+S_5\cdot C'\ge 2/7$, a contradiction.
Thus we can write $S_3\cap S_7=C+\Gamma$, where 
$\Gamma$ is an effective $1$-cycle with $\Supp \Gamma \subset S_5$.
In particular, $P_5\notin S_5$.
The divisor $84A$ is Cartier at $P_3$, $P_4$, and $P_7$.
We get
\[
\frac {9}{5}= 84A\cdot S_3\cdot (S_7-S_4) =84A\cdot \Gamma\in \ZZ,
\]
a contradiction.
\end{proof}

Now the proof of \ref{theorem-main-19} of Theorem \ref{theorem-main}
can be finished similar to \ref{proof-theorem-13}.

\section{Toric Sarkisov links}
\begin{proposition}
Let $X$ be a toric $\QQ$-Fano threefold and let $P\in X$ be a 
cyclic quotient singularity of index $r$. 
Let $f\colon \tilde X\to X$ be the Kawamata blowup of $P\in X$.
Then a general member 
of $|-K_X|$ is a normal surface having at worst Du Val singularities.
The linear system $|-K_{X}|$ has only isolated base points.
In particular, $-K_{\tilde X}$ is nef and big.
The map $f\colon \tilde X\to X$ can be completed by 
a toric Sarkisov link \textup(cf. \eqref{eq3.1}\textup).
\end{proposition}
\begin{proof}
This can be shown by explicit computations in all cases
of Proposition \ref{proposition-main-toric}. Consider, for example, the case
$X=\PP(3,4,5,7)$. Let $x_3$, $x_4$, $x_5$, $x_7$ be quasi-homogeneous coordinates in 
$\PP(3,4,5,7)$. A section $S\in |-K_X|$ is given by a quasi-homogeneous polynomial
of degree $19$. By taking this polynomial as a general linear combination of 
$x_3^5 x_4$,  $x_3^3x_5^2$, $x_3^4 x_7$, 
$x_4 x_5^3$, $x_4^3 x_7$,  
$x_5x_7^2$ we see that the base locus of $|-K_X|$ is the union of 
four coordinate points and the surface $S$ has only quotient singularities.
Since $K_S$ is Cartier, the singularities of $S$ are Du Val.
Further, we can write $K_{\tilde X}+\tilde S=f^*(K_X+S)\sim 0$,
where $\tilde S$ is the proper transform of $S$.
Hence, $\tilde S\in |-K_{\tilde S}|$ and the linear system 
$|-K_{\tilde X}|$ has only isolated base points
outside of $f^{-1}(P)$. In particular, $-K_{\tilde X}$ is nef.
It is easy to check that $-K_{\tilde X}^3>0$,
i.e., $-K_{\tilde X}$ is  big. Recall that $\rho(\tilde X)=2$. 
So, the Mori cone $\operatorname{NE}(\tilde X)$ has exactly two extremal rays,
say $R_1$ and $R_2$. Let $R_1$ is generated by $f$-exceptional curves. 
If $-K_{\tilde X}$ is ample, we run the MMP starting from $R_2$.
Otherwise we make a flop in $R_2$ and run the MMP.
Clearly, we obtain Sarkisov link \eqref{eq3.1}.
\end{proof}
Explicitly, for weighted projective spaces from Proposition \ref{proposition-main-toric},
we have the following diagram of Sarkisov links.
Here an arrow $X_1 \stackrel{\frac1r}{\longrightarrow} X_2$ indicates that 
there is Sarkisov link described above that starts from Kawamata blowup of 
a cyclic quotient singularity of index $r>1$ on $X_1$ and the target variety is 
$X_2$. 
\[
\xymatrix@R=3pc@C=7pc{
\\
\text{$\PP^1$-bundle}&
\PP(1,2,3,5)
\ar[dr]^{\frac13}
\ar[dl]_{\frac12}
\ar[r]^{\frac15}&
\text{\begin{tabular}{c} $\QQ$-conic \\ bundle\end{tabular}}
\\
\PP(1^3,2)
\ar[u]^{\frac12}
&
\PP(3,4,5,7)
\ar[r]|-{\frac15}
\ar@/_1.7pc/[ddl]|-<(.6){\frac17}
\ar@/^1.5pc/[ddl]|-<(.3){\frac14}
\ar[u]^{\frac13}
&
\PP(1^2,2,3)
\ar[u]^{\frac13}
\ar@/^0.7pc/[dd]|-{\frac12}
\\
\PP^3 
&
&
\\
\PP(1,3,4,5)
\ar@/^1.4pc/[uu]^{\frac15}
\ar@/^2pc/[rr]|-{\frac14}
\ar@/_0.4pc/[u]|-{\frac13}
&
\PP(2,3,5,7)
\ar@/^1.4pc/[l]|-{\frac15}
\ar@/_1.4pc/[uul]|-<(.8){\frac13}
\ar@/^1.4pc/[uur]|-{\frac12}
\ar@/_1.4pc/[uur]|-{\frac17}
&
\text{\begin{tabular}{c} del Pezzo \\ fibration\end{tabular}}&
} 
\]


\end{document}